\documentclass[amstex,12pt]{article}
\usepackage{cite}
\usepackage{indentfirst}
\usepackage{amsmath,color}
\usepackage{amssymb,amsmath,amsthm}
\newtheorem{theorem}{Theorem}[section]
\newtheorem{corollary}{Corollary}[section]
\newtheorem{lemma}{Lemma}[section]
\newtheorem{definition}{Definition}[section]
\newtheorem{remark}{Remark}[section]
\newtheorem{proposition}{Proposition}[section]

\newcommand{\beq}{\begin{equation}}
\newcommand{\eeq}{\end{equation}}
\newcommand{\beqn}{\begin{eqnarray}}
\newcommand{\eeqn}{\end{eqnarray}}

\allowdisplaybreaks

\begin{document}
\allowdisplaybreaks

\title{Right fractional Sobolev space via Riemann$-$Liouville derivatives on time scales and an application to fractional boundary value problem on time scales\thanks{This work is supported by the National Natural Science Foundation of China under Grant Nos. 11861072 and 11561071.}}
\author{ Xing Hu and Yongkun Li\thanks{The corresponding author. Email: yklie@ynu.edu.cn.}\\
 Department of Mathematics, Yunnan University\\
Kunming, Yunnan 650091\\
 People's Republic of China}
\date{}
\maketitle{}

\begin{abstract}
Using the concept of fractional derivatives of Riemann$-$Liouville on time scales, we first introduce right fractional Sobolev spaces and characterize them.  Then, we prove the equivalence of some norms in the introduced spaces, and obtain their completeness, reflexivity, separability and some imbeddings. Finally, as an application, we propose  a recent  method to study the existence of weak solutions of fractional boundary value problems on time scales by using variational method and critical point theory, and by constructing an appropriate variational setting, we obtain two existence results of the problem.
\end{abstract}

\section{Introduction}
\setcounter{equation}{0}
In the past two decades, fractional calculus and fractional (order) differential equations have aroused widespread interest and attention in the field of differential equations, as well as in applied mathematics and science. In addition to true mathematical interest and curiosity, this trend is also driven by interesting scientific and engineering applications that have produced fractional differential equation models to better describe (time) memory effects and (space) non-local phenomenon (\cite{13,14,15,16,17}). It is the rise of these applications that revitalize the field of fractional calculus and fractional differential equations and call for further research in this field.

As we all know, discrete-time systems are as important as continuous time systems. Therefore, it is of equal importance to study the solvability of boundary value problems of fractional differential equations and difference equations. Fortunately, the time scale theory proposed by Stefan Hilger (\cite{t1}) can unify the study of differential equations and difference equations. So far, the research of time scale theory has attracted extensive attention all over the world. It can be applied to engineering, physics, economics, population dynamics and other fields (\cite{t4,t5,t7,t8,a1,a2,a3,a4}).

In order to study the existence and multiplicity of solutions of differential equations and difference equations in a unified framework, Refs. (\cite{7,2,2a}) have studied some Sobolev space theories on time scales.
However, so far, there is no right fractional Sobolev space via Riemann$-$Liouville derivatives on time scales. In order to fill this gap, the main purpose of this paper is to establish right fractional Sobolev spaces on time scales through Riemann$-$Liouville fractional derivative, and study some of their basic properties. Then, as an application of our new theory, we study the solvability of a class of fractional boundary value problems on time scales.

\section{Preliminaries}
\setcounter{equation}{0}
In this section, we will recall some basic known notations, definitions, and  results, which are needed in the sequel.

Throughout this paper, we denote by $\mathbb{T}$ a time scale. We will use the following notations:
$J_{\mathbb{R}}^0=[a,b)$,  $J_{\mathbb{R}}=[a,b]$,
 $J^0=J_{\mathbb{R}}^0\cap{\mathbb{T}}$, $J=J_{\mathbb{R}}\cap{\mathbb{T}}$, $J^k=[a,\rho(b)]\cap\mathbb{T}$.

\begin{definition}(\cite{t2})\label{1}
For $t\in\mathbb{T}$ we define the forward jump operator $\sigma:\mathbb{T}\rightarrow\mathbb{T}$ by
$
\sigma(t):=\inf\{s\in\mathbb{T}:s>t\},
$
while the backward jump operator $\rho:\mathbb{T}\rightarrow\mathbb{T}$ is defined by
$
\rho(t):=\sup\{s\in\mathbb{T}:s<t\}.
$
\end{definition}

\begin{remark}(\cite{t2})
\begin{itemize}
  \item [$(1)$]
  In Definition \ref{1}, we put $\inf\varnothing=\sup\mathbb{T}$ (i.e., $\sigma(t)=t$ if $\mathbb{T}$ has a maximum $t$) and $\sup\varnothing=\inf\mathbb{T}$ (i.e., $\rho(t)=t$ if $\mathbb{T}$ has a minimum $t$), where $\emptyset$ denotes the empty set.
  \item [$(2)$]
  If $\sigma(t)>t$, we say that $t$ is right$-$scattered, while if $\rho(t)<t$, we say that $t$ is left$-$scattered. Points that are right$-$scattered and left$-$scattered at the same time are called isolated.
  \item [$(3)$]
  If $t<\sup\mathbb{T}$ and $\sigma(t)=t$, we say that $t$ is right$-$dense, while if $t>\inf\mathbb{T}$ and $\rho(t)=t$, we say that $t$ is left$-$dense. Points that are right$-$dense and left$-$dense at the same time are called dense.
  \item[$(4)$]
  The graininess function $\mu:\mathbb{T}\rightarrow[0,\infty)$ is defined by
  $
  \mu(t):=\sigma(t)-t.
  $
  \item [$(5)$]
  The derivative makes use of the set $\mathbb{T}^k$, which is derived from the time scale $\mathbb{T}$ as follows: If $\mathbb{T}$ has a left$-$scattered maximum $M$, then $\mathbb{T}^k:=\mathbb{T}\backslash\{M\}$; otherwise, $\mathbb{T}^k:=\mathbb{T}$.
\end{itemize}
\end{remark}

\begin{definition}(\cite{t3})
Assume that $f:\mathbb{T} \rightarrow \mathbb{R}$ is a function and let
$t\in \mathbb{T}^{k}$. Then we define $f^{\Delta}(t)$ to be the
number (provided it exists) with the property that given any
$\varepsilon>0$, there is a neighborhood $U$ of
t (i.e, $U=(t-\delta,t+\delta)\cap \mathbb{T}$ for some
$\delta>0$) such that
\[
\vert f(\sigma(t))-f(s)-f^{\Delta}(t)(\sigma(t)-s)\vert\leq
\varepsilon\vert\sigma(t)-s\vert
\] for all $s\in U$.
We call $f^{\Delta}(t)$ the delta (or Hilger) derivative of f at t.
Moreover, we say that $f$ is delta (or Hilger) differentiable (or in
short: differentiable) on $\mathbb{T}^{k}$ provided $f^{\Delta}(t)$
exists for all $t\in \mathbb{T}^{k}$. The function
$f^{\Delta}:\mathbb{T}^{k} \rightarrow \mathbb{R}$ is then called
the (delta) derivative of $f$ on $\mathbb{T}^{k}$.
\end{definition}

\begin{definition}(\cite{t2})
A function $f:\mathbb{T}\rightarrow\mathbb{R}$ is called rd$-$continuous provided it is continuous at right$-$dense points in $\mathbb{T}$ and its left$-$sided limits exist (finite) at left$-$dense points in $\mathbb{T}$. The set of rd$-$continuous functions $f:\mathbb{T}\rightarrow\mathbb{R}$ will be denoted by
$
C_{rd}=C_{rd}(\mathbb{T})=C_{rd}(\mathbb{T},\mathbb{R}).
$
The set of functions $f:\mathbb{T}\rightarrow\mathbb{R}$ that are differentiable and whose derivative is rd$-$continuous is denoted by
$
C_{rd}^1=C_{rd}^1(\mathbb{T})=C_{rd}^1(\mathbb{T},\mathbb{R}).
$
\end{definition}

\begin{definition}(\cite{3})
Let $J$ denote a closed bounded interval in $\mathbb{T}$. A function  $F:J\rightarrow\mathbb{R}$ is called a delta antiderivative of function $f:J^0\rightarrow\mathbb{R}$ provided $F$ is continuous on $J$, delta differentiable at $J^0$, and $F^\Delta(t)=f(t)$ for all $t\in J^0$. Then, we define the $\Delta-$integral of $f$ from $a$ to $b$ by
$
\int_a^bf(t)\Delta t:=F(b)-F(a).
$
\end{definition}

\begin{theorem}(\cite{t3})\label{ft3}
If $a,b\in\mathbb{T}$ and $f,g\in C_{rd}(\mathbb{T})$, then
\begin{eqnarray*}
\int_{J^0}f^\sigma(t)g^\Delta(t)\Delta t=(fg)(b)-(fg)(a)-\int_{J^0}f^\Delta(t)g(t)\Delta t.
\end{eqnarray*}
\end{theorem}

\begin{proposition}(\cite{10})\label{2}
Let $f$ be an increasing continuous function on  $J$. If $F$ is the extension of $f$ to the real interval $J_{\mathbb{R}}$ given by
\begin{equation*}
F(s):=
\left\{
\begin{aligned}
f(s)&,&\quad if\,\,s\in \mathbb{T},\\
f(t)&,&\quad if\,\,s\in (t,\sigma(t))\notin \mathbb{T},
\end{aligned}
\right.
\end{equation*}
then
\begin{eqnarray*}
\int_a^bf(t)\Delta t\leq\int_a^bF(t)dt.
\end{eqnarray*}
\end{proposition}

Motivated by Definition 4 in \cite{a4} and Definition 2.1 in \cite{4}, we can present the right Riemann$-$Liouville fractional integral and derivative on time scales as follows:

\begin{definition}(Fractional integral on time scales)\label{3} Suppose  $h$ is an integrable function on $J$. Let $0<\alpha\leq1$. Then the left fractional integral of order $\alpha$ of $h$ is defined by
\begin{eqnarray*}
_a^{\mathbb{T}}I_t^\alpha h(t):=\int_a^t\frac{(t-\sigma(s))^{\alpha-1}}{\Gamma(\alpha)}h(s)\Delta s.
\end{eqnarray*}
The right fractional integral of order $\alpha$ of $h$ is defined by
\begin{eqnarray}\label{2.5}
_t^{\mathbb{T}}I_b^\alpha h(t):=\int_t^b\frac{(s-\sigma(t))^{\alpha-1}}{\Gamma(\alpha)}h(s)\Delta s,
\end{eqnarray}
where $\Gamma$ is the gamma function.
\end{definition}

\begin{definition}(Riemann$-$Liouville fractional derivative on time scales)\label{4}  Let  $t\in\mathbb{T}$, $0<\alpha\leq1$, and $h:\mathbb{T}\rightarrow\mathbb{R}$. The left Riemann$-$Liouville fractional derivative of order $\alpha$ of $h$ is defined by
\begin{eqnarray*}
_a^{\mathbb{T}}D_t^\alpha h(t):=\bigg({_a^{\mathbb{T}}}I_t^{1-\alpha} h(t)\bigg)^\Delta=\frac{1}{\Gamma(1-\alpha)}\bigg(\int_a^t(t-\sigma(s))^{-\alpha}h(s)\Delta s\bigg)^\Delta.
\end{eqnarray*}
The right Riemann$-$Liouville fractional derivative of order $\alpha$ of $h$ is defined by
\begin{eqnarray}\label{2.6}
_t^{\mathbb{T}}D_b^\alpha h(t):=-\bigg({_t^{\mathbb{T}}}I_b^{1-\alpha} h(t)\bigg)^\Delta=\frac{-1}{\Gamma(1-\alpha)}\bigg(\int_t^b(s-\sigma(t))^{-\alpha}h(s)\Delta s\bigg)^\Delta.
\end{eqnarray}
\end{definition}

Motivated by Definition 4 and Equation (21) in \cite{a4} and Theorem 2.1 in \cite{16}, we can present the right Caputo fractional derivative on time scales as follows:

\begin{definition}(Caputo fractional derivative on time scales)\label{cd}  Let  $t\in\mathbb{T}$, $0<\alpha\leq1$  and $h:\mathbb{T}\rightarrow\mathbb{R}$. The left Caputo fractional derivative of order $\alpha$ of $h$ is defined by
\begin{eqnarray*}
{_a^{\mathbb{T}\,C}}D_t^\alpha h(t):={_a^{\mathbb{T}}}I_t^{1-\alpha} h^\Delta(t)=\frac{1}{\Gamma(1-\alpha)}\int_a^t(t-\sigma(s))^{-\alpha}h^\Delta(s)\Delta s.
\end{eqnarray*}
The right Caputo fractional derivative of order $\alpha$ of $h$ is defined by
\begin{eqnarray*}
{_t^{\mathbb{T}\,C}}D_b^\alpha h(t):=-{_t^{\mathbb{T}}}I_b^{1-\alpha} h^\Delta(t)=\frac{-1}{\Gamma(1-\alpha)}\int_t^b(s-\sigma(t))^{-\alpha}h^\Delta(s)\Delta s.
\end{eqnarray*}
\end{definition}

\begin{definition}(\cite{3'})\label{d2'}
For $f:\mathbb{T}\rightarrow\mathbb{R}$, the time scale or generalized Laplace transform of $f$, denoted by $\mathcal{L}_{\mathbb{T}}\{f\}$ or $F(z)$, is given by
\begin{eqnarray*}
\mathcal{L}_{\mathbb{T}}\{f\}(z)=F(z):=\int_0^\infty f(t)g^\sigma(t)\Delta t,
\end{eqnarray*}
where $g(t)=e_{\ominus z}(t,0)$.
\end{definition}

\begin{theorem}(\cite{3'}) (Inversion formula of the Laplace transform)\label{t1'}
Suppose that $F(z)$ is analytic in the region $Re_\mu(z)>Re_\mu(c)$ and $F(z)\rightarrow0$ uniformly as $\vert z\vert\rightarrow\infty$ in this region. Suppose $F(z)$ has finitely many regressive poles of finite order $\{z_1,z_2,\ldots,z_n\}$ and $\widetilde{F}_{\mathbb{R}}(z)$ is the transform of the function $\widetilde{f}(t)$ on $\mathbb{R}$ that corresponds to the transform $F(z)=F_{\mathbb{T}}(z)$ of $f(t)$ on $\mathbb{T}$. If
\begin{eqnarray*}
\int_{c-i\infty}^{c+i\infty}\vert \widetilde{F}_{\mathbb{R}}(z)\vert\vert dz\vert<\infty,
\end{eqnarray*}
then
\begin{eqnarray*}
f(t)=\sum_{i=1}^nRes_{z=z_i}e_z(t,0)F(z)
\end{eqnarray*}
has transform $F(z)$ for all $z$ with $Re(z)>c$.
\end{theorem}

Motivated by Definition 3.1 in \cite{6'}, we can present the right Riemann$-$Liouville fractional integral on time scales as follows:

\begin{definition}\label{d3'}( Right Riemann$-$Liouville fractional integral on time scales)
Let $\alpha>0$, $\mathbb{T}$ be a time scale, and $f:\mathbb{T}\rightarrow\mathbb{R}$. The right Riemann$-$Liouville fractional integral of $f$ of order $\alpha$ on the time scale $\mathbb{T}$, denoted by $_bI_{\mathbb{T}}^\alpha f$, is defined by
\begin{eqnarray*}
_bI_{\mathbb{T}}^\alpha f(t)=\mathcal{L}_{\mathbb{T}}^{-1}\left[\frac{F(z)}{(-z)^\alpha}\right](t).
\end{eqnarray*}
\end{definition}

\begin{theorem}(\cite{2})\label{8}
A function $f:J\rightarrow\mathbb{R}^N$ is absolutely continuous on $J$ if and only if $f$ is $\Delta$-differentiable $\Delta-a.e.$ on $J^0$ and
\begin{eqnarray*}
f(t)=f(a)+\int_{[a,t)_\mathbb{T}}f^\Delta(s)\Delta s,\quad \forall t\in J.
\end{eqnarray*}
\end{theorem}

\begin{theorem}(\cite{1})\label{9}
A function $f:\mathbb{T}\rightarrow\mathbb{R}$ is absolutely continuous on $\mathbb{T}$ if and only if the following conditions are satisfied:
\begin{itemize}
  \item [$(i)$]
  $f$ is $\Delta-$differentiable $\Delta-a.e.$ on $J^0$ and $f^\Delta\in L^1(\mathbb{T})$.
  \item [$(ii)$]
  The equality
  \begin{eqnarray*}
  f(t)=f(a)+\int_{[a,t)_\mathbb{T}}f^\Delta(s)\Delta s
  \end{eqnarray*}
holds for every  $t\in\mathbb{T}$.
\end{itemize}
\end{theorem}

\begin{theorem}(\cite{5})\label{10}
A function $q:J_{\mathbb{R}}\rightarrow\mathbb{R}^m$ is absolutely continuous if and only if there exist a constant $c\in\mathbb{R}^m$ and a function $\varphi\in L^1$ such that
\begin{eqnarray*}
q(t)=c+(I_{a^+}^1\varphi)(t),\quad t\in J_{\mathbb{R}}.
\end{eqnarray*}
In this case, we have $q(a)=c$ and $q'(t)=\varphi(t)$, $t\in J_{\mathbb{R}}$ a.e..
\end{theorem}

\begin{theorem}(\cite{5}) (integral representation)\label{11}
Let $\alpha\in(0,1]$ and $q\in L^1$. Then, $q$ has a right$-$sided Riemann$-$Liouville derivative $D_{b^-}^\alpha q$ of order $\alpha$ if and only if there exist a constant $d\in\mathbb{R}^m$ and a function $\psi\in L^1$ such that
\begin{eqnarray*}
q(t)=\frac{1}{\Gamma(\alpha)}\frac{d}{(b-t)^{1-\alpha}}+(I_{b^-}^\alpha\psi)(t),\quad t\in J_{R}\quad a.e..
\end{eqnarray*}
In this case, we have $I_{b^-}^{1-\alpha}q(b)=d$ and $(D_{b^-}^\alpha q)(t)=\psi(t)$, $t\in J_{\mathbb{R}}$ a.e..
\end{theorem}

\begin{lemma}(\cite{7})\label{13}
Let $f\in L_\Delta^1(J^0)$. Then, a necessary and sufficient condition for the validity of the equality
\begin{eqnarray*}
\int_{J^0}(f\cdot\varphi^\Delta)(s)\Delta s=0,\quad for\,\,every\,\,\varphi\in C_{0,rd}^1(J^k)
\end{eqnarray*}
is the existence of a constant $c\in\mathbb{R}$ such that
\begin{eqnarray*}
f\equiv c \quad \Delta-a.e. \,\,on\,\, J^0.
\end{eqnarray*}
\end{lemma}

\begin{definition}(\cite{7})\label{14}
Let $p\in \bar{\mathbb{R}}$ be such that $p\geq 1$ and $u:J\rightarrow\bar{\mathbb{R}}$. Say that $u$ belongs to $W_{\Delta}^{1,p}(J)$ if and only if $u\in L_{\Delta}^p(J^0)$ and there exists $g:J^k\rightarrow\bar{\mathbb{R}}$ such that $g\in L_{\Delta}^p(J^0)$ and
\begin{eqnarray*}
\int_{J^0}(u\cdot\varphi^\Delta)(s)\Delta s=-\int_{J^0}(g\cdot\varphi^\sigma)(s)\Delta s,\quad \forall\varphi\in C_{0,rd}^1(J^k),
\end{eqnarray*}
where
\begin{eqnarray*}
C_{0,rd}^1(J^k):=\bigg\{f:J\rightarrow\mathbb{R}:f\in C_{rd}^1(J^k),\,f(a)=f(b)\bigg\}
\end{eqnarray*}
and $C_{rd}^1(J^k)$ is the set of all continuous functions on $J$ such that they are $\Delta-$differential on $J^k$ and their $\Delta-$derivatives are $rd-$continuous on $J^k$.
\end{definition}

\begin{theorem}(\cite{7})\label{15}
Let $p\in\bar{\mathbb{R}}$ be such that $p\geq 1$. Then, the set $L_\Delta^p(J^0)$ is a Banach space together with the norm defined for every $f\in L_\Delta^p(J^0)$ as
\begin{equation*}
\|f\|_{L_\Delta^p}:=
\left\{
\begin{aligned}
&\bigg[\int_{J^0}\vert f\vert^p(s)\Delta s\bigg]^{\frac{1}{p}},&\quad if\,\,p\in \mathbb{R},\\
&\inf\{C\in\mathbb{R}:\vert f\vert\leq C \,\Delta-a.e. \,\, on \,\,J^0\},&\quad if\,\,p=+\infty.
\end{aligned}
\right.
\end{equation*}
Moreover, $L_\Delta^2(J^0)$ is a Hilbert space together with the inner product given for every $(f,g)\in L_\Delta^2(J^0)\times L_\Delta^2(J^0)$ by
\begin{eqnarray*}
(f,g)_{L_\Delta^2}:=\int_{J^0}f(s)\cdot g(s)\Delta s.
\end{eqnarray*}
\end{theorem}

\begin{theorem}(\cite{4})\label{16}
Fractional integration operators are bounded in $L^p(J_{\mathbb{R}})$, i.e., the following estimate
\begin{eqnarray*}
\|I_{a^+}^\alpha\varphi\|_{L^p(a,b)}\leq\frac{(b-a)^{Re\alpha}}{Re\alpha\vert\Gamma(\alpha)\vert}\|\varphi\|_{L^p(J_{\mathbb{R}})},\quad Re\alpha>0
\end{eqnarray*}
holds.
\end{theorem}

\begin{proposition}(\cite{7})\label{17}
Suppose $p\in\bar{\mathbb{R}}$ and $p\geq 1$. Let $p'\in\bar{\mathbb{R}}$ be such that $\frac{1}{p'}+\frac{1}{p'}=1$. Then, if $f\in L_\Delta^p(J^0)$ and $g\in L_\Delta^{p'}(J^0)$, then $f\cdot g\in L_\Delta^1(J^0)$ and
\begin{eqnarray*}
\|f\cdot g\|_{L_\Delta^1}\leq\|f\|_{L_\Delta^p}\cdot\|g\|_{L_\Delta^{p'}}.
\end{eqnarray*}
This expression is called H$\ddot{o}$lder's inequality and Cauchy$-$Schwarz's inequality whenever $p=2$.
\end{proposition}

\begin{theorem}(\cite{8}) (First Mean Value Theorem)\label{18}
Let $f$ and $g$ be bounded and integrable functions on $J$, and let $g$ be nonnegative (or nonpositive) on $J$. Let us set
\begin{eqnarray*}
m=\inf\{f(t):t\in J^0\} \quad and \quad M=\sup\{f(t):t\in J^0\}.
\end{eqnarray*}
Then there exists a real number $\Lambda$ satisfying the inequalities $m\leq \Lambda\leq M$ such that
\begin{eqnarray*}
\int_a^bf(t)g(t)\Delta t=\Lambda\int_a^bg(t)\Delta t.
\end{eqnarray*}
\end{theorem}

\begin{corollary}(\cite{8})\label{19}
Let $f$ be an integrable function on $J$ and let $m$ and $M$ be the infimum and supremum, respectively, of $f$ on $J^0$. Then there exists a number $\Lambda$ between $m$ and $M$ such that
\begin{eqnarray*}
\int_a^bf(t)\Delta t=\Lambda(b-a).
\end{eqnarray*}
\end{corollary}

\begin{theorem}(\cite{8})\label{20}
Let $f$ be a function defined on $J$ and let $c\in\mathbb{T}$ with $a<c<b$. If $f$ is $\Delta-$integrable from $a$ to $c$ and from $c$ to $b$, then $f$ is $\Delta-$integrable from $a$ to $b$ and
\begin{eqnarray*}
\int_a^bf(t)\Delta t=\int_a^cf(t)\Delta t+\int_c^bf(t)\Delta t.
\end{eqnarray*}
\end{theorem}

\begin{lemma}(\cite{11}) (A time scale version of the Arzela$-$Ascoli theorem)\label{21}
Let $X$ be a subset of $C(J,\mathbb{R})$ satisfying the following conditions:
\begin{itemize}
  \item [$(i)$]
  $X$ is bounded.
  \item [$(ii)$]
  For any given $\epsilon>0$, there exists $\delta>0$ such that $t_1,t_2\in J$, $\vert t_1-t_2\vert<\delta$ implies $\vert f(t_1)-f(t_2)\vert<\epsilon$ for all $f\in X$.
\end{itemize}
Then, $X$ is relatively compact.
\end{lemma}

\section{Some fundamental properties of Right Riemann-Liouville fractional operators on time scales}
\setcounter{equation}{0}

Motivated by Theorem 2 in \cite{a4}, we can present and prove the Cauchy result on time scales as follows:

\begin{theorem}(Cauchy Result on Time Scales)\label{t2'}
Let $n\in\{1,2\}$, $\mathbb{T}$ be a time scale with $a,t_1,\ldots,t_n\in\mathbb{T}$, $t_i<b$, $i=1,\ldots,n$, and $f$ an integrable function on $\mathbb{T}$. Then,
\begin{eqnarray*}
\int_{t_n}^b\ldots\int_{t_1}^b f(t_0)\Delta t_0\ldots\Delta t_{n-1}=\frac{1}{(n-1)!} \int_{t}^b(s-\sigma(t))^{n-1}\Delta s.
\end{eqnarray*}
\end{theorem}

\begin{proof}
The proof is similar to the proof of Theorem 2 in \cite{a4}, we will omit it here.
\end{proof}

Motivated by Theorem 1.3 in \cite{4'}, we can present and prove the following Theorem:
\begin{theorem}\label{t3'}
If $x:\mathbb{T}\rightarrow\mathbb{C}$ is regulated and $X(t)=\int_t^0x(\tau)\Delta\tau$ for $t\in\mathbb{T}$, then
\begin{eqnarray*}
\mathcal{L}_{\mathbb{T}}\{X\}(z)=\frac{1}{-z}\mathcal{L}_{\mathbb{T}}\{x\}(z)
\end{eqnarray*}
for all $z\in\mathcal{D}\{x\}\backslash\{0\}$ such that $\lim\limits_{t\rightarrow\infty}\{X(t)e_{\ominus z}(t)\}=0$.
\end{theorem}

\begin{proof}
The proof is similar to the proof of Theorem 1.3 in \cite{4'}, we will omit it here.
\end{proof}

Inspired by \cite{1'}, we can obtain the consistency of Definition \ref{3} and Definition \ref{d3'} by using the above theory of the Laplace transform on time scales and the inverse Laplace transform on time scales.

\begin{theorem}\label{t4'}
Let $\alpha>0$, $\mathbb{T}$ be a time scale, $[a,b]_{\mathbb{T}}$ be an interval of $\mathbb{T}$, and $f$ be an integrable function on $[a,b]_{\mathbb{T}}$. Then, $\left({_t^{\mathbb{T}}I_b^\alpha} f\right)(t)={_bI_{\mathbb{T}}^\alpha} f(t)$.
\end{theorem}

\begin{proof}
Using the Laplace transform on $\mathbb{T}$ for (\ref{2.5}), in view of Definition \ref{3}, Theorem \ref{t2'}, Theorem \ref{t3'} and Definition \ref{d2'}, we have
{\setlength\arraycolsep{2pt}
\begin{eqnarray}\label{e1'}
&&\mathcal{L}_{\mathbb{T}}\left\{\left({_t^{\mathbb{T}}I_b^\alpha} f\right)(t)\right\}(z)\nonumber\\
&=&\mathcal{L}_{\mathbb{T}}\left\{\frac{1}{\Gamma(\alpha)}
\int_t^b(s-\sigma(t))^{\alpha-1}f(s)\Delta s\right\}(z)\nonumber\\
&=&\mathcal{L}_{\mathbb{T}}\left\{\int_{t_\alpha}^b\ldots\int_{t_1}^b f(t_0)\Delta t_0\ldots\Delta t_{n-1}\right\}(z)\nonumber\\
&=&\frac{1}{(-z)^\alpha}\mathcal{L}_{\mathbb{T}}\{f\}(z)\nonumber\\
&=&\frac{F(z)}{(-z)^\alpha}(t).
\end{eqnarray}}
Taking the inverse Laplace transform on $\mathbb{T}$ for (\ref{2.5}), with an eye to Definition \ref{d3'}, one arrives at
\begin{eqnarray*}
\left({_t^{\mathbb{T}}I_b^\alpha} f\right)(t)=\mathcal{L}_{\mathbb{T}}^{-1}\left[\frac{F(z)}{(-z)^\alpha}\right](t)
={_bI_{\mathbb{T}}^\alpha} f(t).
\end{eqnarray*}
The proof is complete.
\end{proof}

Combining with \cite{3}, \cite{6'} and Theorem \ref{t4'}, we can prove that the analogues of Proposition 15, Proposition 16, Proposition 17, Corollary 18, Theorem 20 and Theorem 21 of the right Riemann$-$Liouville fractional operators on time scales remain intact under the new Definition \ref{3}.

\begin{proposition}\label{5}
Let $h$ be $\Delta-$integrable on $J$ and $0<\alpha\leq 1$. Then
$
_t^{\mathbb{T}}D_b^\alpha h(t)=-\Delta\circ{_t^{\mathbb{T}}I_b^{1-\alpha}}h(t).
$
\end{proposition}

\begin{proof}
Let $h:\mathbb{T}\rightarrow\mathbb{R}$. In view of $(\ref{2.5})$  and $(\ref{2.6})$, we obtain
\begin{align*}
{_t^\mathbb{T}}D_b^{\alpha}h(t)
=&\frac{-1}{\Gamma(1-\alpha)}\bigg(\int_t^b(s-\sigma(t))^{-\alpha}h(s)\Delta s\bigg)^\Delta\\
=&-\bigg({_t^{\mathbb{T}}}I_b^{1-\alpha} h(t)\bigg)^\Delta\\
=&-(\Delta\circ{_t^\mathbb{T}}I_b^{1-\alpha})h(t).
\end{align*}
The proof is complete.
\end{proof}

\begin{proposition}\label{6}
Let $h$ be integrable on $J$, then its right  Riemann$-$Liouville $\Delta-$fractional integral satisfies
\begin{eqnarray*}
_t^{\mathbb{T}}I_b^\alpha\circ{_t^{\mathbb{T}}I_b^\beta}={_t^{\mathbb{T}}}I_b^{\alpha+\beta}
={_t^{\mathbb{T}}I_b^\beta}\circ{_t^{\mathbb{T}}I_b^\alpha}
\end{eqnarray*}
for $\alpha>0$ and $\beta>0$.
\end{proposition}

\begin{proof}
Inspired by the proof of Proposition 3.4 in \cite{6'}, with an view of Definition \ref{d3'}, we have
\begin{eqnarray}\label{b1}
_bI_{\mathbb{T}}^\beta({_bI_{\mathbb{T}}^\alpha}f)(t)
=\mathcal{L}_{\mathbb{T}}^{-1}\left[\frac{\mathcal{L}_{\mathbb{T}}
\{{_bI_{\mathbb{T}}^\alpha}f\}}{(-z)^\beta}\right](t)
=\mathcal{L}_{\mathbb{T}}^{-1}\left[\frac{F(s)}{(-s)^{\alpha+\beta}}\right](t)
={_bI_{\mathbb{T}}^{\alpha+\beta}}f(t).
\end{eqnarray}
Combining with (\ref{b1}) and Theorem \ref{t4'}, one gets that
\begin{eqnarray*}
_t^{\mathbb{T}}I_b^\alpha\circ{_t^{\mathbb{T}}I_b^\beta}
={_t^{\mathbb{T}}}I_b^{\alpha+\beta}.
\end{eqnarray*}
In a similarly way, one arrives at
\begin{eqnarray*}
_t^{\mathbb{T}}I_b^\beta\circ{_t^{\mathbb{T}}I_b^\alpha}
={_t^{\mathbb{T}}}I_b^{\alpha+\beta}.
\end{eqnarray*}
Consequently, we obtain that
\begin{eqnarray*}
_t^{\mathbb{T}}I_b^\alpha\circ{_t^{\mathbb{T}}I_b^\beta}={_t^{\mathbb{T}}}I_b^{\alpha+\beta}
={_t^{\mathbb{T}}I_b^\beta}\circ{_t^{\mathbb{T}}I_b^\alpha}.
\end{eqnarray*}
The proof is complete.
\end{proof}

\begin{proposition}\label{173}
If function $h$  is integrable on $J$, then
$
_t^{\mathbb{T}}D_b^\alpha\circ{_t^{\mathbb{T}}I_b^\alpha h}=h.
$
\end{proposition}

\begin{proof}
Taking account of Propositions \ref{5} and \ref{6}, one can get
\begin{eqnarray*}
_t^{\mathbb{T}}D_b^\alpha\circ{_t^{\mathbb{T}}I_b^\alpha h(t)}=-\bigg({_t^{\mathbb{T}}}I_b^{1-\alpha} ({_t^{\mathbb{T}}I_b^\alpha} (h(t))\bigg)^\Delta=-\left({_t^{\mathbb{T}}}I_bh(t)\right)^\Delta=h.
\end{eqnarray*}
The proof is complete.
\end{proof}

\begin{corollary}
For $0<\alpha\leq 1$, we have
\begin{eqnarray*}
_t^{\mathbb{T}}D_b^\alpha\circ{_t^{\mathbb{T}}D_b^{-\alpha}}=Id\quad and \quad_t^{\mathbb{T}}I_b^{-\alpha}\circ{_t^{\mathbb{T}}I_b^\alpha}=Id,
\end{eqnarray*}
where $Id$ denotes the identity operator.
\end{corollary}

\begin{proof}
In view of Proposition \ref{173}, we have
\begin{eqnarray*}
_t^{\mathbb{T}}D_b^\alpha\circ{_t^{\mathbb{T}}D_b^{-\alpha}}
=\,_t^{\mathbb{T}}D_b^\alpha\circ{_t^{\mathbb{T}}I_b^{\alpha}}=Id\quad \mathrm{and} \quad_t^{\mathbb{T}}I_b^{-\alpha}\circ{_t^{\mathbb{T}}I_b^\alpha}
=\,_t^{\mathbb{T}}D_b^{\alpha}\circ{_t^{\mathbb{T}}I_b^\alpha}=Id.
\end{eqnarray*}
The proof is complete.
\end{proof}

\begin{theorem}\label{thm21}
Let $f\in C(J)$ and $\alpha>0$. Then $f\in {_t^{\mathbb{T}}}I_b^\alpha(J)$ if and only if
\begin{eqnarray}\label{t3.1}
_t^{\mathbb{T}}I_b^{1-\alpha}f\in C^1(J)
\end{eqnarray}
and
\begin{eqnarray}\label{t3.2}
\bigg({_t^{\mathbb{T}}}I_b^{1-\alpha}f(t)\bigg)\bigg\vert_{t=b}=0,
\end{eqnarray}
where ${_t^{\mathbb{T}}}I_b^\alpha(J)$ denotes the space of functions that can be represented by the right Riemann$-$Liouville $\Delta$-integral of order $\alpha$ of a $C(J)-$function.
\end{theorem}

\begin{proof}
Suppose $f\in\,_t^{\mathbb{T}}I_b^\alpha(J)$, $f(t)=\,^{\mathbb{T}}_tI_b^\alpha g(t)$ for some $g\in C(J)$, and
\begin{eqnarray*}
_t^{\mathbb{T}}I_b^{1-\alpha}(f(t))=\,_t^{\mathbb{T}}I_b^{1-\alpha}(^{\mathbb{T}}_tI_b^\alpha g(t)).
\end{eqnarray*}
In view of Proposition \ref{6}, one gets
\begin{eqnarray*}
_t^{\mathbb{T}}I_b^{1-\alpha}(f(t))=\,_t^{\mathbb{T}}I_bg(t)=\int_t^bg(s)\Delta s.
\end{eqnarray*}
As a result, $_t^{\mathbb{T}}I_b^{1-\alpha}f\in C(J)$ and
\begin{eqnarray*}
\bigg({_t^{\mathbb{T}}}I_b^{1-\alpha}f(t)\bigg)\bigg\vert_{t=b}=\int_b^bg(s)\Delta s=0.
\end{eqnarray*}
Inversely, suppose that $f\in C(J)$ satisfies (\ref{t3.1}) and (\ref{t3.2}). Then, by applying Taylor's formula  to function $_t^{\mathbb{T}}I_b^{1-\alpha}f$, we obtain
\begin{eqnarray*}
_t^{\mathbb{T}}I_b^{1-\alpha}f(t)=\int_t^b\frac{\Delta}{\Delta s}\,_s^{\mathbb{T}}I_b^{1-\alpha}f(s)\Delta s,\quad\forall t\in J.
\end{eqnarray*}
Let $\varphi(t)=\frac{\Delta}{\Delta t}\,_t^{\mathbb{T}}I_b^{1-\alpha}f(t)$. Note that $\varphi\in C(J)$ by (\ref{t3.1}). Now by Proposition \ref{6}, one sees that
\begin{eqnarray*}
_t^{\mathbb{T}}I_b^{1-\alpha}(f(t))=\,_t^{\mathbb{T}}I_b^1\varphi(t)
=\,_t^{\mathbb{T}}I_b^{1-\alpha}[_t^{\mathbb{T}}I_b^{\alpha}\varphi(t)]
\end{eqnarray*}
and hence
\begin{eqnarray*}
_t^{\mathbb{T}}I_b^{1-\alpha}(f(t))-\,_t^{\mathbb{T}}I_b^{1-\alpha}[_t^{\mathbb{T}}
I_b^{\alpha}\varphi(t)]\equiv0.
\end{eqnarray*}
Therefore, we have
\begin{eqnarray*}
_t^{\mathbb{T}}I_b^{1-\alpha}[f(t)-_t^{\mathbb{T}}
I_b^{\alpha}\varphi(t)]\equiv0.
\end{eqnarray*}
From the uniqueness of solution to Abel's integral equation (\cite{33}), this implies that $f-\,_t^{\mathbb{T}}I_b^{\alpha}\varphi\equiv0$. Hence, $f=\,_t^{\mathbb{T}}I_b^{\alpha}\varphi$ and $f\in\,_t^{\mathbb{T}}I_b^{\alpha}(J)$. The proof is complete.
\end{proof}

\begin{theorem}\label{7}
Let $\alpha>0$ and $f\in C(J)$ satisfy the condition in Theorem \ref{thm21}. Then,
\begin{eqnarray*}
(_a^{\mathbb{T}}I_t^\alpha\circ{_a^{\mathbb{T}}D_t^\alpha})(f)=f.
\end{eqnarray*}
\end{theorem}

\begin{proof}
Combining with Theorem \ref{thm21} and Proposition \ref{173}, we can see that
\begin{eqnarray*}
_t^{\mathbb{T}}I_b^{\alpha}\circ\,_t^{\mathbb{T}}D_b^{\alpha}f(t)
=\,_t^{\mathbb{T}}I_b^{\alpha}\circ\,_t^{\mathbb{T}}D_b^{\alpha}
(_t^{\mathbb{T}}I_b^{\alpha}\varphi(t))=\,_t^{\mathbb{T}}I_b^{\alpha}\varphi(t)=f(t).
\end{eqnarray*}
The proof is complete.
\end{proof}

Motivated by the proof of Equation (2.20) in \cite{4}, we can present and prove the following Theorem:

\begin{theorem}\label{12}
Let $\alpha>0$, $p,q\geq1$, and $\frac{1}{p}+\frac{1}{q}\leq 1+\alpha$, where $p\neq 1$ and $q\neq 1$ in the case when $\frac{1}{p}+\frac{1}{q}=1+\alpha$. Moreover, let
\begin{eqnarray*}
_a^\mathbb{T}I_t^\alpha(L^p):=\bigg\{f:f={_a^\mathbb{T}}I_t^\alpha g,\,g\in L^p(J)\bigg\}
\end{eqnarray*}
and
\begin{eqnarray*}
_t^\mathbb{T}I_b^\alpha(L^p):=\bigg\{f:f={_t^\mathbb{T}}I_b^\alpha g,\,g\in L^p(J)\bigg\}.
\end{eqnarray*}
Then the following integration by parts formulas hold.
\begin{itemize}
  \item [$(a)$]
  If $\varphi\in L^p(J)$ and $\psi\in L^q(J)$, then
  \begin{eqnarray*}
  \int_a^b\varphi(t)\bigg({_a^\mathbb{T}}I_t^\alpha\psi\bigg)(t)\Delta t=\int_a^b\psi(t)\bigg({_t^\mathbb{T}}I_b^\alpha\varphi\bigg)(t)\Delta t.
  \end{eqnarray*}
  \item [$(b)$]
  If $g\in {_t^\mathbb{T}}I_b^\alpha(L^p)$ and $f\in {_a^\mathbb{T}}I_t^\alpha (L^q)$, then
  \begin{eqnarray*}
  \int_a^bg(t)\bigg({_a^\mathbb{T}}D_t^\alpha f\bigg)(t)\Delta t=\int_a^bf(t)\bigg({_t^\mathbb{T}}D_b^\alpha g\bigg)(t)\Delta t.
  \end{eqnarray*}
  \item [$(c)$]
  For Caputo fractional derivatives, if $g\in {_t^\mathbb{T}}I_b^\alpha(L^p)$ and $f\in {_a^\mathbb{T}}I_t^\alpha (L^q)$, then
  \begin{eqnarray*}
  \int_a^bg(t)\bigg({_a^{\mathbb{T}\,C}}D_t^\alpha f\bigg)(t)\Delta t=\left[{_t^\mathbb{T}}I_b^{1-\alpha}g(t)\cdot f(t)\right]\bigg\vert_{t=a}^b+\int_a^bf(\sigma(t))\bigg({_t^\mathbb{T}}D_b^\alpha g\bigg)(t)\Delta t.
  \end{eqnarray*}
  and
  \begin{eqnarray*}
  \int_a^bg(t)\bigg({_t^{\mathbb{T}\,C}}D_b^\alpha f\bigg)(t)\Delta t=\left[{_a^\mathbb{T}}I_t^{1-\alpha}g(t)\cdot f(t)\right]\bigg\vert_{t=a}^b+\int_a^bf(\sigma(t))\bigg({_a^\mathbb{T}}D_t^\alpha g\bigg)(t)\Delta t.
  \end{eqnarray*}
\end{itemize}
\end{theorem}

\begin{proof}
\begin{itemize}
  \item [$(a)$]
It follows from Definition \ref{3} and Fubini's theorem on time scales that
{\setlength\arraycolsep{2pt}
\begin{eqnarray*}
&&\int_a^b\varphi(t)\bigg({_a^\mathbb{T}}I_t^\alpha\psi\bigg)(t)\Delta t\nonumber\\
&=&\int_a^b\varphi(t)\left(\int_a^t\frac{(t-\sigma(s))^{\alpha-1}}{\Gamma(\alpha)}\psi(s)\Delta s\right)\Delta t\nonumber\\
&=&\int_a^b\psi(s)\int_s^b\frac{(t-\sigma(s))^{\alpha-1}}{\Gamma(\alpha)}\varphi(t)\Delta t\Delta s\nonumber\\
&=&\int_a^b\psi(t)\int_t^b\frac{(s-\sigma(t))^{\alpha-1}}{\Gamma(\alpha)}\varphi(s)\Delta s\Delta t\nonumber\\
&=&\int_a^b\psi(t)\bigg({_t^\mathbb{T}}I_b^\alpha\varphi\bigg)(t)\Delta t.
\end{eqnarray*}}
The proof is complete.
\item [$(b)$]
It follows from Definition \ref{4} and Fubini's theorem on time scales that
{\setlength\arraycolsep{2pt}
\begin{eqnarray*}
&&\int_a^bg(t)\bigg({_a^\mathbb{T}}D_t^\alpha f\bigg)(t)\Delta t\nonumber\\
&=&\int_a^bg(t)\left(\frac{1}{\Gamma(1-\alpha)}\bigg(\int_a^t(t-\sigma(s))^{-\alpha}f(s)\Delta s\bigg)^\Delta\right)\Delta t\nonumber\\
&=&\int_a^bf(s)\left(\frac{1}{\Gamma(1-\alpha)}\bigg(\int_s^b(t-\sigma(s))^{-\alpha}g(t)\Delta t\bigg)^\Delta\right)\Delta s\nonumber\\
&=&\int_a^bf(t)\left(\frac{1}{\Gamma(1-\alpha)}\bigg(\int_t^b(s-\sigma(t))^{-\alpha}g(s)\Delta s\bigg)^\Delta\right)\Delta t\nonumber\\
&=&\int_a^bg(t)\bigg({_t^\mathbb{T}}D_b^\alpha f\bigg)(t)\Delta t.
\end{eqnarray*}}
The proof is complete.
\item [$(c)$]
It follows from Definition \ref{cd}, Fubini's theorem on time scales and Theorem \ref{ft3}  that
{\setlength\arraycolsep{2pt}
\begin{eqnarray*}
&&\int_a^bg(t)\bigg({_a^{\mathbb{T}\,C}}D_t^\alpha f\bigg)(t)\Delta t\nonumber\\
&=&\int_a^bg(t)\left(\frac{1}{\Gamma(1-\alpha)}\int_a^t(t-\sigma(s))^{-\alpha}f^\Delta(s)\Delta s\right)\Delta t\nonumber\\
&=&\int_a^bf^\Delta(s)\left(\frac{1}{\Gamma(1-\alpha)}\int_s^b(t-\sigma(s))^{-\alpha}g(t)\Delta t\right)\Delta s\nonumber\\
&=&\int_a^bf^\Delta(t)\left(\frac{1}{\Gamma(1-\alpha)}\int_t^b(s-\sigma(t))^{-\alpha}g(s)\Delta s\right)\Delta t\nonumber\\
\nonumber\\
&=&\left[{_t^\mathbb{T}}I_b^{1-\alpha}g(t)\cdot f(t)\right]\bigg\vert_{t=a}^b-\int_a^bf(\sigma(t))\left(\frac{1}{\Gamma(1-\alpha)}\int_t^b(s-\sigma(t))^{-\alpha}g(s)\Delta s\right)^\Delta\nonumber\\
&=&\left[{_t^\mathbb{T}}I_b^{1-\alpha}g(t)\cdot f(t)\right]\bigg\vert_{t=a}^b+\int_a^bf(\sigma(t))\left(\frac{-1}{\Gamma(1-\alpha)}\int_t^b(s-\sigma(t))^{-\alpha}g(s)\Delta s\right)^\Delta\nonumber\\
&=&\left[{_t^\mathbb{T}}I_b^{1-\alpha}g(t)\cdot f(t)\right]\bigg\vert_{t=a}^b+\int_a^bf(\sigma(t))\bigg({_t^\mathbb{T}}D_b^\alpha g\bigg)(t)\Delta t.
\end{eqnarray*}}
The second relation is obtained in a similar way. The proof is complete.
\end{itemize}
\end{proof}

\section{Fractional Sobolev spaces on time scales and their properties}
\setcounter{equation}{0}

In this section, inspired by the above discussion, we present and prove the following lemmas, propositions and theorems, which are of utmost significance for our main results. In the following, let $0<a<b$. Suppose $a,b\in\mathbb{T}$.

Motivated by Theorems \ref{8}-\ref{11}, we  can propose the following definition.
\begin{definition}
Let $0<\alpha\leq1$. By $AC_{\Delta,b^-}^{\alpha,1}(J,\mathbb{R}^N)$ we denote the set of all functions $f:J\rightarrow\mathbb{R}^N$ that have the representation
\begin{eqnarray}\label{22}
f(t)=\frac{1}{\Gamma(\alpha)}\frac{d}{(b-t)^{1-\alpha}}+\,_t^\mathbb{T}I_b^\alpha
\psi(t),\quad t\in J\quad\Delta-a.e.
\end{eqnarray}
with $d\in\mathbb{R}^N$ and $\psi\in L_\Delta^1$.
\end{definition}

\begin{theorem}\label{23}
Let $0<\alpha\leq1$ and $f\in L_\Delta^1$. Then function $f$ has the right Riemann$-$Liouville derivative $_t^\mathbb{T}D_b^\alpha f$ of order $\alpha$ on the interval $J$ if and only if $f\in AC_{\Delta,b^-}^{\alpha,1}(J,\mathbb{R}^N)$; that is, $f$ has the representation $(\ref{22})$. In such a case,
\begin{eqnarray*}
({_t^\mathbb{T}}I_b^{1-\alpha}f)(b)=d, \quad
({_t^\mathbb{T}}D_b^\alpha f)(t)=\psi(t),\quad t\in J\quad\Delta-a.e.
\end{eqnarray*}
\end{theorem}

\begin{proof}
Let  $f\in L_\Delta^1$ have a right  Riemann$-$Liouville derivative $_t^\mathbb{T}D_b^\alpha f$. This means that $_t^\mathbb{T}I_b^{1-\alpha} f$ is (identified to) an absolutely continuous function. From the integral representation of Theorem \ref{8}, there exist a constant vector $d\in\mathbb{R}^N$ and a function $\psi\in L_\Delta^1$ such that
\begin{eqnarray}\label{24}
({_t^\mathbb{T}}I_b^{1-\alpha}f)(t)=d+(_t^\mathbb{T}I_b^1\psi)(t),\quad t\in J,
\end{eqnarray}
with $({_t^\mathbb{T}}I_b^{1-\alpha}f)(b)=d$ and $-\bigg(({_t^\mathbb{T}}I_b^{1-\alpha}f)(t)\bigg)^\Delta={_t^\mathbb{T}}D_b^{\alpha}f(t)=\psi(t)$, $t\in J\quad\Delta-a.e.$.

By Proposition \ref{6} and applying $_t^\mathbb{T}I_b^\alpha$ to $(\ref{24})$ we obtain
\begin{eqnarray}\label{25}
({_t^\mathbb{T}}I_b^{1}f)(t)=({_t^\mathbb{T}}I_b^{\alpha}d)(t)
+(_t^\mathbb{T}I_b^1{_t^\mathbb{T}}I_b^\alpha\psi)(t),\quad t\in J\quad\Delta-a.e..
\end{eqnarray}
The result follows from the $\Delta-$differentiability of $(\ref{25})$.

Conversely, now, let us assume that $(\ref{22})$ holds true. From Proposition \ref{6} and applying ${_t^\mathbb{T}}I_b^{1-\alpha}$ on $(\ref{22})$ we obtain
\begin{eqnarray*}
({_t^\mathbb{T}}I_b^{1-\alpha}f)(t)=d+
({_t^\mathbb{T}}I_b^1 \psi)(t),\quad t\in J\quad\Delta-a.e.
\end{eqnarray*}
and then, ${_t^\mathbb{T}}I_b^{1-\alpha}f$ has an absolutely continuous representation and $f$ has a right  Riemann $-$Liouville derivative ${_t^\mathbb{T}}D_b^{\alpha}f$.
This completes the proof.
\end{proof}

\begin{remark}
\begin{itemize}
  \item [$(i)$]
  By $AC_{\Delta,b^-}^{\alpha,p}$ $(1\leq p<\infty)$ we denote the set of all functions $f:J\rightarrow\mathbb{R}^N$ possessing representation $(\ref{22})$ with $d\in\mathbb{R}^N$ and $\psi\in L_\Delta^p$.
  \item [$(ii)$]
  It is easy to see that  Theorem \ref{23} implies the following one (for any $1\leq p<\infty$): $f$ has the right Riemann$-$Liouville derivative ${_t^\mathbb{T}}D_b^{\alpha}f\in L_\Delta^p$ if and only if $f\in AC_{\Delta,b^-}^{\alpha,p}$, that is, $f$ has the representation $(\ref{22})$ with $\psi\in L_\Delta^p$.
\end{itemize}
\end{remark}

\begin{definition}
Let $0<\alpha\leq1$ and let $1\leq p<\infty$. By right Sobolev space of order $\alpha$ we will mean the set $W_{\Delta,b^-}^{\alpha,p}=W_{\Delta,b^-}^{\alpha,p}(J,\mathbb{R}^N)$ given by
\begin{eqnarray*}
W_{\Delta,b^-}^{\alpha,p}:=\bigg\{u\in L_\Delta^p;\,\exists\, g\in L_\Delta^p,\, \forall\varphi\in C_{c,rd}^\infty \,\,such \,\,that \int_a^bu(t)\cdot{_a^\mathbb{T}}D_t^{\alpha}\varphi(t)\Delta t=\int_a^bg(t)\cdot\varphi(t)\Delta t\bigg\}.
\end{eqnarray*}
\end{definition}

\begin{remark}
A function $g$ given above will be called the weak right fractional derivative of order $0<\alpha\leq1$ of $u$; let us denote it by $^\mathbb{T}u_{b^-}^\alpha$. The uniqueness of this weak derivative follows from (\cite{7}).
\end{remark}

We have the following characterization of $W_{\Delta,b^-}^{\alpha,p}$.
\begin{theorem}\label{thm32}
If $0<\alpha\leq1$ and $1\leq p<\infty$, then
\begin{eqnarray*}
W_{\Delta,b^-}^{\alpha,p}=AC_{\Delta,b^-}^{\alpha,p}\cap L_\Delta^p.
\end{eqnarray*}
\end{theorem}

\begin{proof}
For one thing, if $u\in AC_{\Delta,b^-}^{\alpha,p}\cap L_\Delta^p$, then from Theorem \ref{23} it follows that $u$ has the derivative ${_t^\mathbb{T}}D_b^{\alpha}u\in L_\Delta^p$. Theorem \ref{12} implies that
\begin{eqnarray*}
\int_a^bu(t)\,{_a^\mathbb{T}}D_t^{\alpha}\varphi(t)\Delta t=\int_a^b({_t^\mathbb{T}}D_b^{\alpha}u)(t)\,\varphi(t)\Delta t
\end{eqnarray*}
for any $\varphi\in C_{c,rd}^\infty$. So, $u\in W_{\Delta,b^-}^{\alpha,p}$ with
\begin{eqnarray*}
^\mathbb{T}u_{b^-}^\alpha=g={_t^\mathbb{T}}D_b^{\alpha}u\in L_\Delta^p.
\end{eqnarray*}

For another thing, now, let us assume that $u\in W_{\Delta,b^-}^{\alpha,p}$, that is, $u\in L_\Delta^p$, and there exists a function $g\in L_\Delta^p$ such that
\begin{eqnarray}\label{26}
\int_a^bu(t){_a^\mathbb{T}}D_t^{\alpha}\varphi(t)\Delta t=\int_a^bg(t)\varphi(t)\Delta t
\end{eqnarray}
for any $\varphi\in C_{c,rd}^\infty$.

To show that $u\in AC_{\Delta,b^-}^{\alpha,p}\cap L_\Delta^p$ it suffices to check (Theorem \ref{23} and definition of $AC_{\Delta,b^-}^{\alpha,p}$) that $u$ possesses the right Riemann$-$Liouville derivative of order $\alpha$, belonging to $L_\Delta^p$, that is,   ${_t^\mathbb{T}}I_b^{1-\alpha}u$ is absolutely continuous on $[a,b]_\mathbb{T}$ and its delta derivative of $\alpha$ order (existing $\Delta-a.e.$ on $J$) belongs to $L_\Delta^p$.

In fact, let $\varphi\in C_{c,rd}^\infty$, then $\varphi\in {_a^\mathbb{T}}D_t^{\alpha}(C_{rd})$ and ${_a^\mathbb{T}}D_t^{\alpha}\varphi=-({_t^\mathbb{T}}I_b^{1-\alpha})^\Delta$. From Theorem \ref{12}   it follows that
\begin{align}\label{27}
\int_a^bu(t){_a^\mathbb{T}}D_t^{\alpha}\varphi(t)\Delta t
=&\int_a^bu(t)({_a^\mathbb{T}}I_t^{1-\alpha}\varphi)^\Delta(t)\Delta t\nonumber\\
=&\int_a^b({_t^\mathbb{T}}D_b^{1-\alpha}{_t^\mathbb{T}}I_b^{1-\alpha}u)(t)
({_a^\mathbb{T}}I_t^{1-\alpha}\varphi)^\Delta(t)\Delta t\nonumber\\
=&\int_a^b({_t^\mathbb{T}}I_b^{1-\alpha}u)(t)(\varphi)^\Delta(t)\Delta t.
\end{align}
In view of $(\ref{26})$ and $(\ref{27})$, we get
\begin{eqnarray*}
\int_a^b({_t^\mathbb{T}}I_b^{1-\alpha}u)(t)\varphi^\Delta(t)\Delta t=\int_a^bg(t)\varphi(t)\Delta t
\end{eqnarray*}
for any $\varphi\in C_{c,rd}^\infty$. So, ${_t^\mathbb{T}}I_b^{1-\alpha}u\in W_{\Delta,b^-}^{1,p}$. Consequently, ${_t^\mathbb{T}}I_b^{1-\alpha}u$ is absolutely continuous and its delta derivative is equal to $\Delta-a.e.$ on $J$ to $g\in L_\Delta^P$.
\end{proof}

From the  proof of Theorem \ref{thm32} and   the uniqueness of the weak fractional derivative the following theorem follows.
\begin{theorem}\label{28}
If $0<\alpha\leq1$ and $1\leq p<\infty$, then the weak left fractional derivative $^\mathbb{T}u_{b^-}^\alpha$ of a function $u\in W_{\Delta,b^-}^{\alpha,p}$ coincides with its right Riemann$-$Liouville fractional derivative ${_t^\mathbb{T}}D_b^{\alpha}u$ $\Delta-a.e.$ on $J$.
\end{theorem}

\begin{remark}\label{29}
\begin{itemize}
  \item [$(1)$]
  If $0<\alpha\leq1$ and $(1-\alpha)p<1$, then $AC_{\Delta,b^-}^{\alpha,p}\subset L_\Delta^p$ and, consequently,
  \begin{eqnarray*}
  W_{\Delta,b^-}^{\alpha,p}=AC_{\Delta,b^-}^{\alpha,p}\cap L_\Delta^p=AC_{\Delta,b^-}^{\alpha,p}.
  \end{eqnarray*}
  \item [$(2)$]
  If $0<\alpha\leq1$ and $(1-\alpha)p\geq1$, then $W_{\Delta,b^-}^{\alpha,p}=AC_{\Delta,b^-}^{\alpha,p}\cap L_\Delta^p$ is the set of all functions belong to $AC_{\Delta,b^-}^{\alpha,p}$ that satisfy the condition $({_t^\mathbb{T}}I_b^{1-\alpha}f)(b)=0$.
\end{itemize}
\end{remark}

By using the definition of $W_{\Delta,b^-}^{\alpha,p}$ with $0<\alpha\leq1$ and Theorem \ref{28}, one can easily prove the following result.

\begin{theorem}\label{3.4}
Let $0<\alpha\leq1$ and $1\leq p<\infty$ and $u\in L_\Delta^p$. Then $u\in W_{\Delta,b^-}^{\alpha,p}$ if and only if there exists a function $g\in L_\Delta^p$ such that
\begin{eqnarray*}
\int_a^bu(t){_a^\mathbb{T}}D_t^{\alpha}\varphi(t)\Delta t=\int_a^bg(t)\varphi(t)\Delta t,\quad \varphi\in C_{c,rd}^\infty.
\end{eqnarray*}
In such a case, there exists the right Riemann$-$Liouville derivative ${_t^\mathbb{T}}D_b^{\alpha}u$ of $u$ and $g={_t^\mathbb{T}}D_b^{\alpha}u$.
\end{theorem}

\begin{remark}
Function $g$ will be called the weak right fractional derivative of order $\alpha$ of $u\in W_{\Delta,b^-}^{\alpha,p}$. Its uniqueness follows from \cite{7}. From the above theorem it follows that it coincides with the appropriate Riemann$-$Liouville derivative.
\end{remark}

Let us fix $0<\alpha\leq1$ and consider in the space $W_{\Delta,b^-}^{\alpha,p}$ a norm $\|\cdot\|_{W_{\Delta,b^-}^{\alpha,p}}$ given by
\begin{eqnarray*}
\|u\|_{W_{\Delta,b^-}^{\alpha,p}}^p=\|u\|_{L_\Delta^p}^p+\|_t^\mathbb{T}D_b^\alpha u\|_{L_\Delta^p}^p,\quad u\in W_{\Delta,b^-}^{\alpha,p}.
\end{eqnarray*}
(Here $\|\cdot\|_{L_\Delta}^p$ denotes the delta norm in $L_\Delta^p$ (Theorem \ref{15})).

\begin{lemma}\label{30}
Let $0<\alpha\leq1$ and $1\leq p<\infty$, then
\begin{eqnarray*}
\|_t^\mathbb{T}I_b^\alpha \varphi\|_{L_\Delta^p}^p\leq K^p\|\varphi\|_{L_\Delta^P}^P,
\end{eqnarray*}
where $K=\frac{(b-a)^\alpha}{\Gamma(\alpha+1)}$. i.e., the fractional integration operator is bounded in $L_\Delta^p$.
\end{lemma}

\begin{proof}
The conclusion follows from Theorem \ref{16}, Proposition \ref{17} and Proposition \ref{2}. The proof is complete.
\end{proof}

\begin{theorem}\label{thm35}
If $0<\alpha\leq1$, then the norm $\|\cdot\|_{W_{\Delta,b^-}^{\alpha,p}}$ is equivalent to the norm $\|\cdot\|_{b, W_{\Delta,b^-}^{\alpha,p}}$ given by
\begin{eqnarray*}
\|u\|_{b,W_{\Delta,b^-}^{\alpha,p}}^p=\vert{_t^\mathbb{T}I_b^{1-\alpha}}u(b)\vert^p
+\|\,_t^\mathbb{T}D_b^\alpha u\|_{L_\Delta^p}^p,\quad u\in W_{\Delta,b^-}^{\alpha,p}.
\end{eqnarray*}
\end{theorem}

\begin{proof}

  $(1)$
  Assume that $(1-\alpha)p<1$. On the one hand, for $u\in W_{\Delta,b^-}^{\alpha,p}$ given by
  \begin{eqnarray*}
  u(t)=\frac{1}{\Gamma(\alpha)}\frac{d}{(b-t)^{1-\alpha}}+\,_t^\mathbb{T}I_b^\alpha\varphi(t)
  \end{eqnarray*}
  with $d\in\mathbb{R}^N$ and $\psi\in L_\Delta^p$. Since $(b-t)^{(\alpha-1)p}$ is an increasing monotone function, by using Proposition \ref{2}, we can write that $\int_a^b(b-t)^{(\alpha-1)p}\Delta t\leq \int_a^b(b-t)^{(\alpha-1)p}dt$. And taking into account Lemma \ref{30}, we have
  \begin{align*}
  \|u\|_{L_\Delta^p}^p
  =&\int_a^b\bigg\vert\frac{1}{\Gamma(\alpha)}\frac{d}{(b-t)^{1-\alpha}}
  +\,_a^\mathbb{T}I_t^\alpha\psi(t)\bigg\vert^p\Delta t\\
  \leq&2^{p-1}\bigg(\frac{\vert d\vert^p}{\Gamma^p(\alpha)}\left\vert\int_a^b(b-t)^{(\alpha-1)p}\Delta t\right\vert+\|\,_t^\mathbb{T}I_b^\alpha\varphi\|_{L_\Delta^p}^p\bigg)\\
  \leq&2^{p-1}\bigg(\frac{\vert d\vert^p}{\Gamma^p(\alpha)}\left\vert\int_a^b(b-t)^{(\alpha-1)p}dt\right\vert +\|\,_t^\mathbb{T}I_b^\alpha\psi\|_{L_\Delta^p}^p\bigg)\\
  \leq&2^{p-1}\bigg(\frac{\vert d\vert^p}{\Gamma^p(\alpha)}\frac{1}{(\alpha-1)p+1}(b-a)^{(\alpha-1)p+1} +K^p\|\psi\|_{L_\Delta^p}^p\bigg),
  \end{align*}
  where $K$ is defined in Lemma \ref{30}. Noting that $d={_a^\mathbb{T}}I_t^{1-\alpha}u(b)$, $\varphi=-{_a^\mathbb{T}}D_t^\alpha u$, thus, one obtains
  \begin{align*}
  \|u\|_{L_\Delta^p}^p
  \leq&L_{\alpha,0}(\vert d\vert^p+\|\psi\|_{L_\Delta^p}^p)\\
  \leq&L_{\alpha,0}\bigg(\vert{_t^\mathbb{T}}I_b^{1-\alpha}u(b)\vert^p+\|{_t^\mathbb{T}}D_b^\alpha u\|_{L_\Delta^p}^p\bigg)\\
  =&L_{\alpha,0}\|u\|_{b,W_{\Delta,b^-}^{\alpha,p}}^p,
  \end{align*}
  where
  \begin{eqnarray*}
  L_{\alpha,0}=2^{p-1}\bigg(\frac{(b-a)^{1-(1-\alpha)p}}{\Gamma^p(\alpha)(1-(1-\alpha)p)} +K^p\bigg).
  \end{eqnarray*}
  Consequently,
  \begin{align*}
  \|u\|_{W_{\Delta,a^+}^{\alpha,p}}^p
  =&\|u\|_{L_\Delta^P}^P+\|_a^\mathbb{T}D_t^\alpha u\|_{L_\Delta^p}^p\\
  \leq&L_{\alpha,1}\|u\|_{a,W_{\Delta,b^-}^{\alpha,p}}^p,
  \end{align*}
  where $L_{\alpha,1}=L_{\alpha,0}+1$.

  On the other hand, now, we will prove that there exists a constant $M_{\alpha,1}$ such that
  \begin{eqnarray*}
  \|u\|_{b,W_{\Delta,b^-}^{\alpha,p}}^p\leq M_{\alpha,1}\|u\|_{W_{\Delta,b^-}^{\alpha,p}}^p,
  \quad u\in W_{\Delta,b^-}^{\alpha,p}.
  \end{eqnarray*}
  Indeed, let $u\in W_{\Delta,b^-}^{\alpha,p}$ and consider  coordinate functions $({_t^\mathbb{T}}I_b^{1-\alpha}u)^i$ of ${_t^\mathbb{T}}I_b^{1-\alpha}u$ with  $i\in\{1,\ldots,N\}$. Lemma \ref{30}, Theorem \ref{18} and Corollary \ref{19} imply that there exist constants
  \begin{eqnarray*}
  \Lambda_i\in
  \bigg[\inf_{t\in[a,b)_\mathbb{T}}({_t^\mathbb{T}}I_b^{1-\alpha}u)^i(t),
  \sup_{t\in[a,b)_\mathbb{T}}({_t^\mathbb{T}}I_b^{1-\alpha}u)^i(t)\bigg],
  \end{eqnarray*}
  such that
  \begin{eqnarray*}
  \Lambda_i=\frac{1}{b-a}\int_a^b
  ({_t^\mathbb{T}}I_b^{1-\alpha}u)^i(s)\Delta s
  \end{eqnarray*}
  Hence, if, for all $i=1,2,\ldots,N$,   $({_t^\mathbb{T}}I_b^{1-\alpha}u)^i(t_0)\neq0$, then we can take constants $\theta_i$ such that

  \begin{eqnarray*}
  \theta_i({_t^\mathbb{T}}I_b^{1-\alpha}u)^i(t_0)=\Lambda_i=\frac{1}{b-a}\int_a^b
  ({_t^\mathbb{T}}I_b^{1-\alpha}u)^i(s)\Delta s
  \end{eqnarray*}
  for fixed $t_0\in J^0$. Therefore, we have
  \begin{eqnarray*}
  ({_t^\mathbb{T}}I_b^{1-\alpha}u)^i(t_0)=\frac{\theta_i}{b-a}\int_a^b
  ({_t^\mathbb{T}}I_b^{1-\alpha}u)^i(s)\Delta s.
  \end{eqnarray*}
  From the absolute continuity (Theorem \ref{9}) of $({_t^\mathbb{T}}I_b^{1-\alpha}u)^i$ it follows that
  \begin{eqnarray*}
  ({_t^\mathbb{T}}I_b^{1-\alpha}u)^i(t)=({_t^\mathbb{T}}I_b^{1-\alpha}u)^i(t_0)
  +\int_{[t_0,t)_\mathbb{T}}\bigg[({_t^\mathbb{T}}I_b^{1-\alpha}u)^i(s)\bigg]^\Delta \Delta s
  \end{eqnarray*}
  for any $t\in J$. Consequently, combining with Proposition \ref{5} and Lemma \ref{30}, we see that
  \begin{align*}
  \vert({_t^\mathbb{T}}I_b^{1-\alpha}u)^i(t)\vert
  =&\bigg\vert({_t^\mathbb{T}}I_b^{1-\alpha}u)^i(t_0)
  +\int_{[t_0,t)_\mathbb{T}}\bigg[({_t^\mathbb{T}}I_b^{1-\alpha}u)^i(s)\bigg]^\Delta \Delta s\bigg\vert\\
  \leq&\frac{\vert\theta_i\vert}{b-a}
  \|{_t^\mathbb{T}}I_b^{1-\alpha}u\|_{L_\Delta^1}+\int_{[t_0,t)_\mathbb{T}}
  \vert({_t^\mathbb{T}}D_b^{\alpha}u)(s)\vert\Delta s\\
  \leq&\frac{\vert\theta_i\vert}{b-a}
  \|{_t^\mathbb{T}}I_b^{1-\alpha}u\|_{L_\Delta^1}
  +\|{_t^\mathbb{T}}D_b^{\alpha}u\|_{L_\Delta^1}\\
  \leq&\frac{\vert\theta_i\vert}{b-a}\frac{(b-a)^{1-\alpha}}{\Gamma(2-\alpha)}
  \|u\|_{L_\Delta^1}+\|{_t^\mathbb{T}}D_b^{\alpha}u\|_{L_\Delta^1}\\
  \end{align*}
  for $t\in J$. In particular,
  \begin{eqnarray*}
  \vert({_t^\mathbb{T}}I_b^{1-\alpha}u)^i(b)\vert\leq\frac{\vert\theta_i\vert}{b-a}\frac{(b-a)^{1-\alpha}}
  {\Gamma(2-\alpha)}\|u\|_{L_\Delta^1}+\|{_t^\mathbb{T}}D_b^{\alpha}u\|_{L_\Delta^1}.
  \end{eqnarray*}
  So,
  \begin{align*}
  \vert({_t^\mathbb{T}}I_b^{1-\alpha}u)(b)\vert
  \leq&N\bigg(\frac{\vert\theta\vert(b-a)^{-\alpha}}
  {\Gamma(2-\alpha)}+1\bigg)\left(\|u\|_{L_\Delta^1}
  +\|{_t^\mathbb{T}}D_b^{\alpha}u\|_{L_\Delta^1}\right)\\
  \leq& NM_{\alpha,0}(b-a)^{\frac{p-1}{p}}\left(\|u\|_{L_\Delta^p}
  +\|{_t^\mathbb{T}}D_b^{\alpha}u\|_{L_\Delta^p}\right),
  \end{align*}
  where $\vert\theta\vert=\max\limits_{i\in\{1,2,\ldots,N\}}\vert\theta_i\vert$ and $M_{\alpha,0}=\frac{|\theta|(b-a)^{-\alpha}}{\Gamma(2-\alpha)}+1$. Thus,
  \begin{align*}
  |({_t^\mathbb{T}}I_b^{1-\alpha}u)(b)|^p
  \leq&N^pM^p_{\alpha,0}(b-a)^{p-1}2^{p-1}\left(\|u\|^p_{L_\Delta^p}
  +\|{_t^\mathbb{T}}D_b^{\alpha}u\|^p_{L_\Delta^p}\right),
  \end{align*}
  and, consequently,
  \begin{align*}
  \|u\|_{b,W_{\Delta,b^-}^{\alpha,p}}^p
  =&\vert{_t^\mathbb{T}I_b^{1-\alpha}}u(b)\vert^p+\|_t^\mathbb{T}D_b^\alpha u\|_{L_\Delta^p}^p\\
  \leq&\bigg(N^pM^p_{\alpha,0}(b-a)^{p-1}2^{p-1}+1\bigg)\left(\|u\|^p_{L_\Delta^p}
  +\|{_t^\mathbb{T}}D_b^{\alpha}u\|^p_{L_\Delta^p}\right)\\
  =&M_{\alpha,1}\|u\|_{W_{\Delta,b^-}^{\alpha,p}}^p,
  \end{align*}
  where $M_{\alpha,1}=N^pM^p_{\alpha,0}(b-a)^{p-1}2^{p-1}+1$.

  If some of or even all of $({_t^\mathbb{T}}I_b^{1-\alpha}u)^i(t_0)=0$,  from the above proof process, we can see that our conclusion is still valid.

   $(2)$
  When $(1-\alpha)p\geq1$, then (Remark \ref{29}) $W_{\Delta,b^-}^{\alpha,p}=AC_{\Delta,b^-}^{\alpha,p}\cap L_\Delta^p$ is the set of all functions belong to $AC_{\Delta,b^-}^{\alpha,p}$ that satisfy the condition $({_t^\mathbb{T}}I_b^{1-\alpha}u)(b)=0$. Consequently, in the same way as in the case of $(1-\alpha)p<1$ (putting $d=0$), we obtain the inequality
  \begin{eqnarray*}
  \|u\|_{W_{\Delta,b^-}^{\alpha,p}}^p\leq L_{\alpha,1}\|u\|_{b,W_{\Delta,b^-}^{\alpha,p}}^p,\quad \mathrm{with}\,\,\mathrm{some}\,\,L_{\alpha,1}>0.
  \end{eqnarray*}
  The inequality
  \begin{eqnarray*}
  \|u\|_{b,W_{\Delta,b^-}^{\alpha,p}}^p\leq M_{\alpha,1}\|u\|_{W_{\Delta,b^-}^{\alpha,p}}^p,\quad \mathrm{with}\,\,\mathrm{some}\,\,M_{\alpha,1}>0
  \end{eqnarray*}
  is obvious (it is sufficient to put $M_{\alpha,1}=1$ and use the fact that $({_a^\mathbb{T}}I_t^{1-\alpha}u)(b)=0$).
 The proof is complete.
\end{proof}

We are now in a position to state and prove  some basic properties of the introduced space.
\begin{theorem}\label{3.6}
The space $W_{\Delta,b^-}^{\alpha,p}$ is complete with respect to each of the norms $\|\cdot\|_{W_{\Delta,b^-}^{\alpha,p}}$ and $\|\cdot\|_{b, W_{\Delta,b^-}^{\alpha,p}}$ for any $0<\alpha\leq1$ and $1\leq p<\infty$.
\end{theorem}

\begin{proof}In view of Theorem \ref{thm35}, we only need to  show that $W_{\Delta,b^-}^{\alpha,p}$ with the norm $\|\cdot\|_{b, W_{\Delta,b^-}^{\alpha,p}}$ is complete. Let $\{u_k\}\subset W_{\Delta,b^-}^{\alpha,p}$ be a Cauchy sequence with respect to this norm. So, the sequences $\{{_t^\mathbb{T}}I_b^{1-\alpha}u_k(b)\}$  and $\{_t^\mathbb{T}D_b^\alpha u_k\}$ are   Cauchy sequences in   $\mathbb{R}^N$ and $L_\Delta^p$, respectively.

Let $d\in\mathbb{R}^N$ and $\psi\in L_\Delta^p$ be the limits of the above sequences in $\mathbb{R}^N$ and $L_\Delta^p$, respectively. Then the function
\begin{eqnarray*}
u(t)=\frac{d}{\Gamma(\alpha)}{(b-t)^{\alpha-1}}+{_t^\mathbb{T}}I_b^\alpha\psi(t),\quad t\in J\quad\Delta-a.e.
\end{eqnarray*}
belongs to $W_{\Delta,b^-}^{\alpha,p}$ and is the limit of $\{u_k\}$ in $W_{\Delta,b^-}^{\alpha,p}$ with respect to $\|\cdot\|_{b, W_{\Delta,b^-}^{\alpha,p}}$. (To assert that $u\in L_\Delta^p$ it is sufficient to consider the cases $(1-\alpha)p<1$ and $(1-\alpha)p\geq1$. In the second case ${_t^\mathbb{T}}I_b^{1-\alpha}u_k(b)=0$ for any $k\in\mathbb{N}$ and, consequently, $d=0$.) The proof is complete.
\end{proof}

In the proofs of the next two theorems we use the method presented in \cite{9}
$P_{121}$ Proposition 8.1 $(b), (c)$.
\begin{theorem}\label{31}
The space $W_{\Delta,b^-}^{\alpha,p}$ is reflexive with respect to the norm $\|\cdot\|_{W_{\Delta,b^-}^{\alpha,p}}$ for any $0<\alpha\leq1$ and $1<p<\infty$.
\end{theorem}

\begin{proof}
Let us consider $W_{\Delta,b^-}^{\alpha,p}$ with the norm $\|\cdot\|_{W_{\Delta,b^-}^{\alpha,p}}$ and define a mapping
\begin{eqnarray*}
\lambda:W_{\Delta,b^-}^{\alpha,p}\ni u\mapsto\left(u,\, _t^\mathbb{T}D_b^\alpha u\right)\in L_\Delta^p\times L_\Delta^p.
\end{eqnarray*}
It is obvious that
\begin{eqnarray*}
\|u\|_{W_{\Delta,b^-}^{\alpha,p}}=\|\lambda u\|_{L_\Delta^p\times L_\Delta^p},
\end{eqnarray*}
where
\begin{eqnarray*}
\|\lambda u\|_{L_\Delta^p\times L_\Delta^p}=\bigg(\sum_{i=1}^2\|(\lambda u)_i\|_{L_\Delta^p}^p\bigg)^{\frac{1}{p}},\quad \lambda u=\left(u,\, _t^\mathbb{T}D_b^\alpha u\right)\in L_\Delta^p\times L_\Delta^p,
\end{eqnarray*}
which means that the operator $\lambda:u\mapsto\left(u,\, _t^\mathbb{T}D_b^\alpha u\right)$ is an isometric isomorphic mapping and the space $W_{\Delta,b^-}^{\alpha,p}$ is isometric isomorphic to the space $\Omega=\bigg\{\left(u,\, _t^\mathbb{T}D_b^\alpha u\right):\forall u\in W_{\Delta,b^-}^{\alpha,p}\bigg\}$, which is a closed subset of $L_\Delta^p\times L_\Delta^p$ as $W_{\Delta,b^-}^{\alpha,p}$ is closed.

Since $L_\Delta^p$ is reflexive, the Cartesian product space $L_\Delta^p\times L_\Delta^p$ is also a reflexive space with respect to the norm $\|v\|_{L_\Delta^p\times L_\Delta^p}=\bigg(\sum\limits_{i=1}^2\|v_i\|_{L_\Delta^p}^p\bigg)^{\frac{1}{p}}$, where $v=\left(v_1,\, v_2\right)\in L_\Delta^p\times L_\Delta^p$.

Thus, $W_{\Delta,b^-}^{\alpha,p}$ is reflexive with respect to the norm $\|\cdot\|_{W_{\Delta,b^-}^{\alpha,p}}$.
\end{proof}

\begin{theorem}
The space $W_{\Delta,b^-}^{\alpha,p}$ is separable with respect to the norm $\|\cdot\|_{W_{\Delta,b^-}^{\alpha,p}}$ for any $0<\alpha\leq1$ and $1\leq p<\infty$.
\end{theorem}

\begin{proof}
Let us consider $W_{\Delta,b^-}^{\alpha,p}$ with the norm $\|\cdot\|_{W_{\Delta,b^-}^{\alpha,p}}$ and the mapping $\lambda$ defined in the proof of Theorem \ref{31}. Obviously, $\lambda(W_{\Delta,b^-}^{\alpha,p})$ is separable as a subset of separable space $L_\Delta^p\times L_\Delta^p$. Since $\lambda$ is the isometry, $W_{\Delta,b^-}^{\alpha,p}$ is also separable with respect to the norm $\|\cdot\|_{W_{\Delta,b^-}^{\alpha,p}}$.
\end{proof}

\begin{proposition}\label{034}
Let $0<\alpha\leq1$ and $1<p<\infty$. For all $u\in W_{\Delta,b^-}^{\alpha,p}$, if $1-\alpha\geq\frac{1}{p}$ or $\alpha>\frac{1}{p}$, then
\begin{eqnarray}\label{33}
\|u\|_{L_\Delta^p}\leq\frac{b^\alpha}{\Gamma(\alpha+1)}\left\|\,_t^\mathbb{T}D_b^\alpha u\right\|_{L_\Delta^p}.
\end{eqnarray}
If $\alpha>\frac{1}{p}$ and $\frac{1}{p}+\frac{1}{q}=1$, then
\begin{eqnarray}\label{34}
\|u\|_\infty\leq\frac{b^{\alpha-\frac{1}{p}}}{\Gamma(\alpha)((\alpha-1)q+1)^{\frac{1}{q}}}
\left\|\,_t^\mathbb{T}D_b^\alpha u\right\|_{L_\Delta^p}.
\end{eqnarray}
\end{proposition}

\begin{proof}
In view of Remark \ref{29} and Theorem \ref{7}, in order to prove inequalities $(\ref{33})$ and $(\ref{34})$, we only need to prove that
\begin{eqnarray}\label{35}
\left\|_t^\mathbb{T}I_b^\alpha(_t^\mathbb{T}D_b^\alpha u) \right\|_{L_\Delta^p}\leq\frac{b^\alpha}{\Gamma(\alpha+1)}\left\|\,_t^\mathbb{T}D_b^\alpha u\right\|_{L_\Delta^p}
\end{eqnarray}
for $1-\alpha\geq\frac{1}{p}$ or $\alpha>\frac{1}{p}$, and
\begin{eqnarray}\label{36}
\left\|\,_t^\mathbb{T}I_b^\alpha(_t^\mathbb{T}D_b^\alpha u) \right\|_\infty\leq\frac{b^{\alpha-\frac{1}{p}}}{\Gamma(\alpha)((\alpha-1)q+1)
^{\frac{1}{q}}}\left\|\,_t^\mathbb{T}D_b^\alpha u\right\|_{L_\Delta^p}
\end{eqnarray}
for $\alpha>\frac{1}{p}$ and $\frac{1}{p}+\frac{1}{q}=1$.

Firstly, we note that $_t^\mathbb{T}D_b^\alpha u\in L_\Delta^p([a,b]_\mathbb{T}, \mathbb{R}^N)$, the inequality $(\ref{35})$ follows from Lemma \ref{30} directly.

We are now in a position to prove $(\ref{36})$. For $\alpha>\frac{1}{p}$, choose $q$ such that $\frac{1}{p}+\frac{1}{q}=1$. For all $u\in W_{\Delta,b^-}^{\alpha,p}$, since $(s-\sigma(t))^{(\alpha-1)q}$ is an increasing monotone function, by using Proposition \ref{2}, we find that $\int_t^b(s-\sigma(t))^{(\alpha-1)q}\Delta s\leq \int_t^b(s-t)^{(\alpha-1)q}ds$. Taking into account of Proposition \ref{17},  we have
\begin{align*}
\left\vert{_t^\mathbb{T}I_b^\alpha(\,_t^\mathbb{T}D_b^\alpha} u(t))\right\vert
=&\frac{1}{\Gamma(\alpha)}\bigg\vert\int_t^b(s-\sigma(t))^{\alpha-1}{\,_t^\mathbb{T}}D_b^\alpha u(s)\Delta s\bigg\vert\\
\leq&\frac{1}{\Gamma(\alpha)}\bigg(\int_t^b(s-\sigma(t))^{(\alpha-1)q}\Delta s\bigg)^{\frac{1}{q}}\|{_t^\mathbb{T}}D_b^\alpha u\|_{L_\Delta^p}\\
\leq&\frac{1}{\Gamma(\alpha)}\bigg(\int_t^b(s-t)^{(\alpha-1)q}ds\bigg)
^{\frac{1}{q}}\|{_t^\mathbb{T}}D_b^\alpha u\|_{L_\Delta^p}\\
\leq&\frac{b^{\frac{1}{q}+\alpha-1}}{\Gamma(\alpha)((\alpha-1)q+1)
^{\frac{1}{q}}}\left\|_t^\mathbb{T}D_b^\alpha u\right\|_{L_\Delta^p}\\
=&\frac{b^{\alpha-\frac{1}{p}}}{\Gamma(\alpha)((\alpha-1)q+1)
^{\frac{1}{q}}}\left\|_t^\mathbb{T}D_b^\alpha u\right\|_{L_\Delta^p}.
\end{align*}
This completes the proof.
\end{proof}

\begin{remark}
\begin{itemize}
  \item [$(i)$]
  According to $(\ref{33})$, we can consider $W_{\Delta,b^-}^{\alpha,p}$ with respect to the norm
  \begin{eqnarray}\label{37}
  \|u\|_{W_{\Delta,b^-}^{\alpha,p}}=\|\,_t^\mathbb{T}D_b^\alpha u\|_{L_\Delta^p}=\bigg(\int_a^b\left\vert{_t^\mathbb{T}D_b^\alpha} u(t)\right\vert^p\Delta t\bigg)^{\frac{1}{p}},
  \end{eqnarray}
  in the following analysis.
  \item [$(ii)$]
  It follows from $(\ref{33})$ and $(\ref{34})$ that $W_{\Delta,b^-}^{\alpha,p}$ is continuously immersed into $C(J, \mathbb{R}^N)$ with the natural norm $\|\cdot\|_\infty$.
\end{itemize}
\end{remark}

\begin{proposition}\label{39}
Let $0<\alpha\leq1$ and $1<p<\infty$. Assume that $\alpha>\frac{1}{p}$ and the sequence $\{u_k\}$ converges weakly to $u$ in $W_{\Delta,b^-}^{\alpha,p}$. Then, $u_k\rightarrow u$ in $C(J, \mathbb{R}^N)$, i.e., $\|u-u_k\|_\infty=0$, as $k\rightarrow\infty$.
\end{proposition}

\begin{proof}
If $\alpha>\frac{1}{p}$, then by $(\ref{34})$ and $(\ref{37})$, the injection of $W_{\Delta,b^-}^{\alpha,p}$ into $C(J, \mathbb{R}^N)$, with its natural norm $\|\cdot\|_\infty$, is continuous, i.e., $u_k\rightarrow u$ in $W_{\Delta,b^-}^{\alpha,p}$, then $u_k\rightarrow u$ in $C(J, \mathbb{R}^N)$.

Since $u_k\rightharpoonup u$ in $W_{\Delta,b^-}^{\alpha,p}$, it follows that $u_k\rightharpoonup u$ in $C(J, \mathbb{R}^N)$. In fact, for any $h\in\left(C(J, \mathbb{R}^N)\right)^*$, if $u_k\rightarrow u$ in $W_{\Delta,b^-}^{\alpha,p}$, then $u_k\rightarrow u$ in $C(J, \mathbb{R}^N)$, and thus $h(u_k)\rightarrow h(u)$. Therefore, $h\in\left(W_{\Delta,b^-}^{\alpha,p}\right)^*$, which means that $\left(C(J, \mathbb{R}^N)\right)^*\subset\left(W_{\Delta,b^-}^{\alpha,p}\right)^*$. Hence, if $u_k\rightharpoonup u$ in $W_{\Delta,b^-}^{\alpha,p}$, then for any $h\in\left(C(J, \mathbb{R}^N)\right)^*$, we have $h\in\left(W_{\Delta,b^-}^{\alpha,p}\right)^*$, and thus $h(u_k)\rightarrow h(u)$, i.e., $u_k\rightharpoonup u$ in $C(J, \mathbb{R}^N)$.

By the Banach$-$Steinhaus theorem, $\{u_k\}$ is bounded in $W_{\Delta,b^-}^{\alpha,p}$ and, hence, in $C(J, \mathbb{R}^N)$. We are now in a position to prove that the sequence $\{u_k\}$ is equi$-$continuous. Let $\frac{1}{p}+\frac{1}{q}=1$ and $t_1, t_2\in[a,b]_\mathbb{T}$, $t_1\leq t_2$, $\forall\,f\in L_\Delta^p(J, \mathbb{R}^N)$, by using Proposition \ref{17}, Proposition \ref{2}, Theorem \ref{20}, and noting $\alpha>\frac{1}{p}$, we have
{\setlength\arraycolsep{2pt}
\begin{eqnarray}\label{38}
&&\left\vert{_{t_1}^\mathbb{T}I_{b}^\alpha} f(t_1)-{_{t_2}^\mathbb{T}}I_{b}^\alpha f(t_2)\right\vert\nonumber\\
&=&\frac{1}{\Gamma(\alpha)}\bigg\vert\int_{t_1}^b(s-\sigma(t_1))^{\alpha-1}f(s)\Delta s-\int_{t_2}^b(s-\sigma(t_2))^{\alpha-1}f(s)\Delta s\bigg\vert\nonumber\\
&\leq&\frac{1}{\Gamma(\alpha)}\bigg\vert\int_{t_1}^b(s-\sigma(t_1))^{\alpha-1}f(s)\Delta s-\int_{t_1}^b(s-\sigma(t_2))^{\alpha-1}f(s)\Delta s\bigg\vert\nonumber\\
&&+\frac{1}{\Gamma(\alpha)}\bigg\vert\int_{t_1}^{t_2}(s-\sigma(t_2))^{\alpha-1}f(s)\Delta s\bigg\vert\nonumber\\
&\leq&\frac{1}{\Gamma(\alpha)}\left\vert\int_{t_1}^b\left((s-\sigma(t_1))^{\alpha-1}
-(s-\sigma(t_2))^{\alpha-1}\right)\right\vert\vert f(s)\vert\Delta s\nonumber\\
&&+\frac{1}{\Gamma(\alpha)}\left\vert\int_{t_1}^{t_2}(s-\sigma(t_2))^{\alpha-1}\right\vert\vert f(s)\vert\Delta s\nonumber\\
&\leq&\frac{1}{\Gamma(\alpha)}\left\vert\bigg(\int_{t_1}^b\left((s-\sigma(t_1))^{\alpha-1}
-(s-\sigma(t_2))^{\alpha-1}\right)^q\Delta s\bigg)^{\frac{1}{q}}\right\vert\|f\|_{L_\Delta^p}\nonumber\\
&&+\frac{1}{\Gamma(\alpha)}\left\vert\bigg(\int_{t_1}^{t_2}(s-\sigma(t_2))^{(\alpha-1)q}\Delta s\bigg)^{\frac{1}{q}}\right\vert\|f\|_{L_\Delta^p}\nonumber\\
&\leq&\frac{1}{\Gamma(\alpha)}\left\vert\bigg(\int_{t_1}^b\left((s-\sigma(t_1))^{(\alpha-1)q}
-(s-\sigma(t_2))^{(\alpha-1)q}\right)\Delta s\bigg)^{\frac{1}{q}}\right\vert\|f\|_{L_\Delta^p}\nonumber\\
&&+\frac{1}{\Gamma(\alpha)}\left\vert\bigg(\int_{t_1}^{t_2}(s-\sigma(t_2))^{(\alpha-1)q}\Delta s\bigg)^{\frac{1}{q}}\right\vert\|f\|_{L_\Delta^p}\nonumber\\
&\leq&\frac{1}{\Gamma(\alpha)}\left\vert\bigg(\int_{t_1}^b\left((s-t_1)^{(\alpha-1)q}
-(s-t_2)^{(\alpha-1)q}\right)ds\bigg)^{\frac{1}{q}}\right\vert\|f\|_{L_\Delta^p}\nonumber\\
&&+\frac{1}{\Gamma(\alpha)}\left\vert\bigg(\int_{t_1}^{t_2}(s-t_2)^{(\alpha-1)q}ds\bigg)
^{\frac{1}{q}}\right\vert\|f\|_{L_\Delta^p}\nonumber\\
&=&\frac{\|f\|_{L_\Delta^p}}{\Gamma(\alpha)\left(1+(\alpha-1)q\right)^{\frac{1}{q}}}
\left\vert\bigg((b-t_1)^{(\alpha-1)q+1}-(b-t_2)^{(\alpha-1)q+1}+(t_2-t_1)^{(\alpha-1)q+1}\bigg)
^{\frac{1}{q}}\right\vert\nonumber\\
&&+\frac{\|f\|_{L_\Delta^p}}{\Gamma(\alpha)\left(1+(\alpha-1)q\right)^{\frac{1}{q}}}
\bigg((t_2-t_1)^{(\alpha-1)q+1}\bigg)^{\frac{1}{q}}\nonumber\\
&\leq&\frac{\|f\|_{L_\Delta^p}}{\Gamma(\alpha)\left(1+(\alpha-1)q\right)^{\frac{1}{q}}}
\bigg((b-t_2)^{(\alpha-1)q+1}-(b-t_1)^{(\alpha-1)q+1}+(t_2-t_1)^{(\alpha-1)q+1}\bigg)
^{\frac{1}{q}}\nonumber\\
&&+\frac{\|f\|_{L_\Delta^p}}{\Gamma(\alpha)\left(1+(\alpha-1)q\right)^{\frac{1}{q}}}
\bigg((t_2-t_1)^{(\alpha-1)q+1}\bigg)^{\frac{1}{q}}\nonumber\\
&\leq&\frac{\|f\|_{L_\Delta^p}}{\Gamma(\alpha)\left(1+(\alpha-1)q\right)^{\frac{1}{q}}}
\bigg((t_2-t_1)^{(\alpha-1)q+1}\bigg)^{\frac{1}{q}}\nonumber\\
&&+\frac{\|f\|_{L_\Delta^p}}{\Gamma(\alpha)\left(1+(\alpha-1)q\right)^{\frac{1}{q}}}
\bigg((t_2-t_1)^{(\alpha-1)q+1}\bigg)^{\frac{1}{q}}\nonumber\\
&=&\frac{2\|f\|_{L_\Delta^p}}{\Gamma(\alpha)\left(1+(\alpha-1)q\right)^{\frac{1}{q}}}
(t_2-t_1)^{\alpha-\frac{1}{p}}.
\end{eqnarray}}
Therefore, the sequence $\{u_k\}$ is equi$-$continuous since, for $t_1, t_2\in[a,b]_\mathbb{T}$, $t_1\leq t_2$, by applying $(\ref{38})$ and in view of $(\ref{37})$, we have
\begin{align*}
\left\vert u_k(t_1)-u_k(t_2)\right\vert
=&\left\vert{_{t_1}^\mathbb{T}I_b^\alpha}(_{t_1}^\mathbb{T}D_{b}^\alpha u_k(t_1))-{_{t_2}^\mathbb{T}}I_b^\alpha(_{t_2}^\mathbb{T}D_{b}^\alpha u_k(t_2))\right\vert\\
\leq&\frac{2(t_2-t_1)^{\alpha-\frac{1}{p}}}{\Gamma(\alpha)\left(1+(\alpha-1)q\right)^{\frac{1}{q}}}
\|_t^\mathbb{T}D_b^\alpha u_k\|_{L_\Delta^p}\\
=&\frac{2(t_2-t_1)^{\alpha-\frac{1}{p}}}{\Gamma(\alpha)\left(1+(\alpha-1)q\right)^{\frac{1}{q}}}
\|_t^\mathbb{T}D_b^\alpha u_k\|_{L_\Delta^p}\\
\leq&\frac{2(t_2-t_1)^{\alpha-\frac{1}{p}}}{\Gamma(\alpha)((\alpha-1)q+1)
^{\frac{1}{q}}}\left\|_t^\mathbb{T}D_b^\alpha u\right\|_{L_\Delta^p}\\
=&\frac{2(t_2-t_1)^{\alpha-\frac{1}{p}}}{\Gamma(\alpha)((\alpha-1)q+1)
^{\frac{1}{q}}}\left\|u_k\right\|_{W_{\Delta,b^-}^{\alpha,p}}\\
\leq&C(t_2-t_1)^{\alpha-\frac{1}{p}},
\end{align*}
where $\frac{1}{p}+\frac{1}{q}=1$ and $C\in\mathbb{R}^+$ is a constant. By the Ascoli$-$Arzela theorem on time scales (Lemma \ref{21}),  $\{u_k\}$ is relatively compact in $C(J, \mathbb{R}^N)$. By the uniqueness of the weak limit in $C(J, \mathbb{R}^N)$, every uniformly convergent subsequence of $\{u_k\}$ converges uniformly on $J$ to $u$.
\end{proof}

\begin{remark}
It follows from Proposition \ref{39} that $W_{\Delta,b^-}^{\alpha,p}$ is compactly immersed into $C(J, \mathbb{R}^N)$ with the natural norm $\|\cdot\|_\infty$.
\end{remark}

\begin{theorem}\label{3.9}
Let $1<p<\infty$, $\frac{1}{p}<\alpha\leq1$, $\frac{1}{p}+\frac{1}{q}=1$, $L:J\times\mathbb{R}^N\times\mathbb{R}^N\rightarrow\mathbb{R}$, $(t,x,y)\mapsto L(t,x,y)$ satisfies
\begin{itemize}
  \item [$(i)$]
  For each $(x,y)\in\mathbb{R}^N\times\mathbb{R}^N$, $L(t,x,y)$ is $\Delta-$measurable in $t$;
  \item [$(ii)$]
  For $\Delta-$almost every $t\in J$, $L(t,x,y)$ is continuously differentiable in $(x,y)$.
\end{itemize}
If there exists $m_1\in C(\mathbb{R}^+,\mathbb{R}^+)$, $m_2\in L_\Delta^1(J,\mathbb{R}^+)$ and $m_3\in L_\Delta^q(J,\mathbb{R}^+)$, $1<q<\infty$, such that, for $\Delta-$a.e. $t\in J$ and every $(x,y)\in\mathbb{R}^N\times\mathbb{R}^N$, one has
\begin{align*}
\vert L(t,x,y)\vert&\leq m_1(\vert x\vert)(m_2(t)+\vert y\vert^p),\\
\vert D_xL(t,x,y)&\vert\leq m_1(\vert x\vert)(m_2(t)+\vert y\vert^p),\\
\vert D_yL(t,x,y)&\vert\leq m_1(\vert x\vert)(m_3(t)+\vert y\vert^{p-1}).
\end{align*}
Then the functional $\varphi$ defined by
\begin{eqnarray*}
\varphi(u)=\int_a^bL(t,u(t),\,_t^\mathbb{T}D_b^\alpha u(t))\Delta t
\end{eqnarray*}
is continuously differentiable on $W_{\Delta,b^-}^{\alpha,p}$, and for all $u,v\in W_{\Delta,b^-}^{\alpha,p}$, we have
\begin{eqnarray}\label{32}
\langle\varphi'(u),v\rangle=\int_a^b\bigg[\left(D_xL(t,u(t),\,_t^\mathbb{T}D_b^\alpha u(t),v(t)\right)+\left(D_yL(t,u(t),\,_t^\mathbb{T}D_b^\alpha u(t),\,_t^\mathbb{T}D_b^\alpha v(t)\right)\bigg]\Delta t.
\end{eqnarray}
\end{theorem}

\begin{proof}
It suffices to prove that $\varphi$ has at every point $u$, a directional derivative $\varphi'(u)\in(W_{\Delta,b^-}^{\alpha,p})^*$ given by $(\ref{32})$ and that the mapping
\begin{eqnarray*}
\varphi':W_{\Delta,b^-}^{\alpha,p}\ni u\mapsto\varphi'(u)\in(W_{\Delta,b^-}^{\alpha,p})^*
\end{eqnarray*}
is continuous. The rest of the proof is similar to the proof of \cite{12} $P_{10}$ Theorem 1.4. We omit it here. The proof is complete.
\end{proof}

\section{An application}
\setcounter{equation}{0}
In this section, we present a recent approach via variational methods and critical point theory to get the existence of weak solutions for the following fractional boundary value problem (FBVP for short) on time scales
\begin{equation}\label{4.1}
\begin{cases}
^{\mathbb{T}}_aD^\alpha_t\left(^{\mathbb{T}}_tD^\alpha_bu(t)\right)=\nabla F(t,u(t)),\quad \Delta-a.e.\,\,t\in J,\\
u(a)=u(b)=0,
\end{cases}
\end{equation}
where $^{\mathbb{T}}_tD^\alpha_b$ and $^{\mathbb{T}}_aD^\alpha_t$ are the right and the left Riemann$-$Liouville fractional derivative operators of order $0<\alpha\leq1$ defined on $\mathbb{T}$ respectively, $F:J\times\mathbb{R}^N\rightarrow\mathbb{R}$ satisfies the following assumption:
\begin{itemize}
  \item [$(H_1)$]
   $F(t,x)$ is $\Delta-$measurable in $t$ for each $x\in\mathbb{R}^N$, continuously differentiable in $x$ for $\Delta-$a.e. $t\in J$ and there exist $m_1\in C(\mathbb{R}^+,\mathbb{R}^+)$ and $m_2\in L_\Delta^1(J,\mathbb{R}^+)$ such that
  \begin{eqnarray*}
  \vert F(t,x)\vert\leq m_1(\vert x\vert)m_2(t),
  \end{eqnarray*}
  \begin{eqnarray*}
  \vert\nabla F(t,x)\vert\leq m_1(\vert x\vert)m_2(t),
  \end{eqnarray*}
  for all $x\in\mathbb{R}^N$ and $\Delta-$a.e. $t\in J$, and $\nabla F(t,x)$ is the gradient of $F$ at $x$.

  \end{itemize}

By constructing a variational structure on $W^{\alpha,2}_{\Delta,b^-}$, we can reduce the problem of finding weak solutions of (\ref{4.1}) to one of seeking the critical points of a corresponding functional.

In particular, when $\mathbb{T}=\mathbb{R}$, FBVP \eqref{4.1} reduces to the standard fractional boundary value problem of the following form
\begin{equation*}
\begin{cases}
_aD^\alpha_t\left(_tD^\alpha_bu(t)\right)=\nabla F(t,u(t)),\quad a.e.\,\,t\in J_{\mathbb{R}},\\
u(a)=u(b)=0.
\end{cases}
\end{equation*}

When $\alpha=1$, FBVP (\ref{4.1}) reduces to the second order Hamiltonian system on time scale $\mathbb{T}$
\begin{equation*}
\begin{cases}
u^{\Delta^2}(t)=\nabla F(\sigma(t),u^\sigma(t)),\quad \Delta-a.e.\,\,t\in J^{\kappa^2},\\
u(a)-u(b)=0,\quad\ u^\Delta(a)-u^\Delta(b)=0.
\end{cases}
\end{equation*}

Although many excellent results have been worked out on the existence of solutions for fractional boundary value problems (\cite{20}$-$\cite{26}) and the second order Hamiltonian systems on time scale $\mathbb{T}$ (\cite{27}$-$\cite{31}), it seems that no similar results were obtained in the literature for FBVP (\ref{4.1}) on time scales. The present section is to show that the critical point theory is an effective approach to deal with the existence of solutions for FBVP (\ref{4.1}) on time scales.

By Theorem \ref{3.6}, the space $W^{\alpha,2}_{\Delta,b^-}$ with the inner product
\begin{eqnarray*}
\langle u,v\rangle=\langle u,v\rangle_{W^{\alpha,2}_{\Delta,b^-}}=\int_a^b(u(t),v(t))\Delta t+\int_a^b({^\mathbb{T}_tD^\alpha_bu(t)},{^\mathbb{T}_tD^\alpha_bv(t)})\Delta t
\end{eqnarray*}
and the induced norm
\begin{eqnarray*}
\|u\|=\|u\|_{W^{\alpha,2}_{\Delta,b^-}}=\left(\int_a^b|u(t)|^2\Delta t+\int_a^b\vert{^\mathbb{T}_tD^\alpha_bu(t)}\vert^2\Delta t\right)^{\frac{1}{2}}
\end{eqnarray*}
is a Hilbert space.

Consider the functional $\varphi:W^{\alpha,2}_{\Delta,b^-}\rightarrow\mathbb{R}$ defined by
\begin{eqnarray}\label{7.1}
\varphi(u)=\frac{1}{2}\int_a^b\vert{^\mathbb{T}_tD^\alpha_b}u(t)\vert^2\Delta t-\int_a^bF(t,u(t))\Delta t,\quad \forall u\in W^{\alpha,2}_{\Delta,b^-}.
\end{eqnarray}

From now on, $\varphi$ which we defined in (\ref{7.1}) will be considered as a functional on $W^{\alpha,2}_{\Delta,b^-}$ with $\frac{1}{2}<\alpha\leq1$. We have the following facts.

\begin{theorem}\label{4.1}
The functional $\varphi$ is continuously differentiable on $W^{\alpha,2}_{\Delta,b^-}$ and
\begin{eqnarray*}
\langle\varphi'(u),v\rangle=\int_a^b\left[(^\mathbb{T}_tD^\alpha_bu(t),
\,^\mathbb{T}_tD^\alpha_bv(t))-(\nabla F(t,u(t)),v(t))\right]\Delta t
\end{eqnarray*}
for all $v\in W^{\alpha,2}_{\Delta,b^-}$.
\end{theorem}

\begin{proof}
Let $L(t,x,y)=\frac{1}{2}\vert y\vert^2-F(t,x)$ for all $x,y\in \mathbb{R}^N$ and $t\in J$. Then, by condition $\mathbf{(H_1)}$, $L(t,x,y)$ meets all the requirements of Theorem \ref{3.9}. Therefore, by Theorem \ref{3.9} it follows  that the functional $\varphi$ is continuously differentiable on $W_{\Delta,b^-}^{\alpha,p}$ and
\begin{eqnarray*}
\langle\varphi'(u),v\rangle=\int_a^b\left[(^\mathbb{T}_tD^\alpha_bu(t),
\,^\mathbb{T}_tD^\alpha_bv(t))-(\nabla F(t,u(t)),v(t))\right]\Delta t
\end{eqnarray*}
for all $v\in W^{\alpha,2}_{\Delta,b^-}$. The proof is complete.
\end{proof}

\begin{definition}\label{5.1}
A function $u:J\rightarrow\mathbb{R}^N$ is called a solution of FBVP $(\ref{4.1})$ if
\begin{itemize}
  \item [$(i)$]
  $^\mathbb{T}_aD^{\alpha-1}_t(^\mathbb{T}_tD_b^\alpha u(t))$ and $^\mathbb{T}_tD^{\alpha-1}_bu(t)$ are differentiable for $\Delta-$a.e. $t\in J^0$ and
  \item [$(ii)$]
  $u$ satisfies FBVP $(\ref{4.1})$.
\end{itemize}
\end{definition}

For a solution $u\in W_{\Delta,b^-}^{\alpha,2}$ of FBVP \eqref{4.1} such that $\nabla F(\cdot,u(\cdot))\in L_\Delta^1(J,\mathbb{R}^N)$, multiplying FBVP $(\ref{4.1})$ by $v\in C_{0,rd}^\infty(J,\mathbb{R}^N)$ yields
{\setlength\arraycolsep{2pt}
\begin{eqnarray}\label{6.1}
&&\int_a^b\left[^{\mathbb{T}}_aD^\alpha_t(^{\mathbb{T}}_tD^\alpha_bu(t),\,v(t))-\nabla F(t,u(t))\right]\Delta t\nonumber\\
&=&\int_a^b\left[(^\mathbb{T}_tD^\alpha_bu(t),\,^\mathbb{T}_tD^\alpha_bv(t))\Delta t-\nabla F(t,u(t))\right]\Delta t\nonumber\\
&=&0,
\end{eqnarray}
after applying $(\mathbf{b})$ of Theorem \ref{12} and Definition \ref{5.1}. Hence, we can give the definition of weak solution for FBVP \eqref{4.1} as follows.

\begin{definition}\label{5.2}
By a weak solution for FBVP $(\ref{4.1})$, we mean that a function $u\in W^{\alpha,2}_{\Delta,b^-}$ such that $\nabla F(\cdot,u(\cdot))\in L_\Delta^1(J,\mathbb{R}^N)$ and satisfies (\ref{6.1}) for all $v\in C_{0,rd}^\infty(J,\mathbb{R}^N)$.
\end{definition}

By our above remarks, any solution $u\in W^{\alpha,2}_{\Delta,b^-}$ of FBVP $(\ref{4.1})$ is a weak solution provided that $\nabla F(\cdot,u(\cdot))\in L_\Delta^1(J,\mathbb{R}^N)$. Our task is now to establish a variational structure on $W^{\alpha,2}_{\Delta,b^-}$ with $\alpha\in\left(\frac{1}{2},1\right]$, which enables us to reduce the existence of weak solutions of FBVP $(\ref{4.1})$ to the one of finding critical points of the corresponding functional.

\begin{theorem}\label{8.1}
If $\frac{1}{2}<\alpha\leq1$, $u\in W^{\alpha,2}_{\Delta,b^-}$ is a critical point of $\varphi$ in $W^{\alpha,2}_{\Delta,b^-}$, i.e., $\varphi'(u)=0$, then $u$ is a weak solution of system \eqref{4.1} with $\frac{1}{2}<\alpha\leq1$.
\end{theorem}

\begin{proof}
Because of $\varphi'(u)=0$, it follows from Theorem \ref{4.1} that
\begin{eqnarray*}
\int_a^b\left[(^\mathbb{T}_tD^\alpha_bu(t),\,^\mathbb{T}_tD^\alpha_bv(t))-\nabla F(t,u(t))\right]\Delta t=0
\end{eqnarray*}
for all $v\in W^{\alpha,2}_{\Delta,b^-}$, and hence for all $v\in C_0^\infty(J,\mathbb{R}^N)$. Therefore, according to Definition \ref{5.2}, $u$ is a weak solution of FBVP (\ref{4.1}) and the proof is complete.
\end{proof}

According to Theorem \ref{8.1}, we see that in order to find weak solutions of FBVP (\ref{4.1}), it suffices to obtain the critical points of the functional $\varphi$ given by (\ref{7.1}). We need to use some critical point theorems. For the reader's convenience, we  present some necessary definitions and theorems and skip the proofs.

Let $H$ be a real Banach space and $C^1(H,\mathbb{R}^N)$ denote the set of functionals that are Fr$\acute{e}$chet differentiable and their Fr$\acute{e}$chet derivatives are continuous on $H$.

\begin{definition}(\cite{32})\label{05.2}
Let $\psi\in C^1(H,\mathbb{R}^N)$. If any sequence $\{u_k\}\subset H$ for which $\{\psi(u_k)\}$ is bounded and $\psi'(u_k)\rightarrow0$ as $k\rightarrow\infty$ possesses a convergent subsequence, then we say $\psi$ satisfies Palais$-$Smale condition (denoted by P.S. condition for short).
\end{definition}

\begin{theorem}(\cite{12})\label{04.1}
Let $H$ be a real reflexive Banach space. If the functional $\psi:H\rightarrow\mathbb{R}^N$ is weakly lower semi$-$continuous and coercive, i.e., $\lim\limits_{\|z\|\rightarrow\infty}\psi(z)=+\infty$, then there exists $z_0\in H$ such that $\psi(z_0)=\inf\limits_{z\in H}\psi(z)$. Moreover, if $\psi$ is also Fr$\acute{e}$chet differentiable on $H$, then $\psi'(z_0)=0$.
\end{theorem}

\begin{theorem}(\cite{32}) (Mountain pass theorem)\label{08.1}
Let $H$ be a real Banach space and  $\psi\in C^1(H,\mathbb{R}^N)$ satisfying P.S. condition. Assume that
\begin{itemize}
  \item [$(i)$]
  $\psi(0)=0$,
  \item [$(ii)$]
  there exist $\rho>0$ and $\sigma>0$ such that $\psi(z)\geq\sigma$ for all $z\in H$ with $\|z\|=\rho$,
  \item [$(iii)$]
  there exists $z_1$ in $H$ with $\|z_1\|\geq\rho$ such that $\psi(z_1)<\sigma$.
\end{itemize}
Then $\psi$ possesses a critical value $c\geq\sigma$. Moreover, $c$ can be characterized as
\begin{eqnarray*}
c=\inf_{g\in\overline{\Omega}}\max_{z\in g([0,1])}\psi(z),
\end{eqnarray*}
where $\overline{\Omega}=\{g\in C([0,1],H):g(0)=0,g(1)=z_1\}$.
\end{theorem}

First, we use Theorem \ref{04.1} to solve the existence of weak solutions for FBVP $(\ref{4.1})$. Suppose that the assumption $(\mathbf{H_1})$ is satisfied. Recall that, in our setting in (\ref{7.1}), the corresponding functional $\varphi$ on $W^{\alpha,2}_{\Delta,b^-}$ given by
\begin{eqnarray*}
\varphi(u)=\frac{1}{2}\int_a^b\vert{^\mathbb{T}_tD^\alpha_b}u(t)\vert^2\Delta t-\int_a^bF(t,u(t))\Delta t,
\end{eqnarray*}
is continuously differentiable according to Theorem \ref{4.1} and is also weakly lower semi$-$continuous functional on $W^{\alpha,2}_{\Delta,b^-}$ as the sum of a convex continuous function and  a weakly continuous function.

In fact, in view of Proposition \ref{39}, if $u_k\rightharpoonup u$ in $W^{\alpha,2}_{\Delta,b^-}$, then $u_k\rightarrow u$ in $C(J,\mathbb{R}^N)$. As a result, $F(t,u_k(t))\rightarrow F(t,u(t))$ $\Delta-$a.e. $t\in[a,b]_{\mathbb{T}}$. Using the Lebesgue dominated convergence theorem, we obtain $\int_a^bF(t,u_k(t))\Delta t\rightarrow \int_a^bF(t,u(t))\Delta t$, which means that the functional $u\rightarrow\int_a^bF(t,u(t))\Delta t$ is weakly continuous on $W^{\alpha,2}_{\Delta,b^-}$. Furthermore, because fractional derivative operator on $\mathbb{T}$ is linear operator, the functional $u\rightarrow\int_a^b\vert{^\mathbb{T}_aD^\alpha_t}u(t)\vert^2\Delta t$ is convex and continuous on $W^{\alpha,2}_{\Delta,b^-}$.

If $\varphi$ is coercive, by Theorem \ref{04.1}, $\varphi$ has a minimum so that FBVP (\ref{4.1}) is solvable. It remains to find conditions under which $\varphi$ is coercive on $W^{\alpha,2}_{\Delta,b^-}$, i.e., $\lim\limits_{\|z\|\rightarrow\infty}\varphi(z)=+\infty$, for $u\in W^{\alpha,2}_{\Delta,b^-}$. We shall know that it suffices to require that $F(t,x)$ is bounded by a function for $\Delta-$a.e. $t\in J$ and all $x\in\mathbb{R}^N$.

\begin{theorem}\label{080.1}
Let $\frac{1}{2}<\alpha\leq1$, and suppose that $F$ satisfies $(\mathbf{H_1})$. If
\begin{eqnarray}\label{008.1}
\vert F(t,x)\vert\leq\overline{a}\vert x\vert^2+\overline{b}(t)\vert x\vert^{2-\gamma}+\overline{c}(t),\quad \Delta-a.e.\,\, t\in J,\,x\in\mathbb{R}^N,
\end{eqnarray}
where $\overline{a}\in\left[0,\frac{\Gamma^2(\alpha+1)}{2b^{2\alpha}}\right)$, $\gamma\in(0,2)$, $\overline{b}\in L_\Delta^{\frac{2}{\gamma}}([a,b]_{\mathbb{T}},\mathbb{R})$ and $\overline{c}\in L_\Delta^1(J,\mathbb{R})$, then FBVP \eqref{4.1} has at least one weak solution which minimizes $\varphi$ on $W^{\alpha,2}_{\Delta,b^-}$.
\end{theorem}

\begin{proof}
Taking account of the arguments above, our task reduces to testify that $\varphi$ is coercive on $W^{\alpha,2}_{\Delta,b^-}$. For $u\in W^{\alpha,2}_{\Delta,b^-}$, it follows from (\ref{008.1}), (\ref{33}) and the H$\ddot{o}$lder inequality on time scales that
{\setlength\arraycolsep{2pt}
\begin{eqnarray*}
&&\varphi(u)\nonumber\\
&=&\frac{1}{2}\int_a^b|^\mathbb{T}_tD^\alpha_bu(t)|^2\Delta t-\int_a^bF(t,u(t))\Delta t\\
&\geq&\frac{1}{2}\int_a^b|^\mathbb{T}_tD^\alpha_bu(t)|^2\Delta t-\overline{a}\int_a^b|u(t)|^2\Delta t-\int_a^b\overline{b}(t)|u(t)|^{2-\gamma}\Delta t-\int_a^b\overline{c}(t)\Delta t\\
&\geq&\frac{1}{2}\|u\|^2-\overline{a}\|u\|_{L_\Delta^2}^2-\left(\int_a^b|\overline{b}(t)|^{\frac{2}{\gamma}}\Delta t\right)^{\frac{\gamma}{2}}\left(\int_a^b|u(t)|^2\Delta t \right)^{1-\frac{\gamma}{2}}\Delta t-\int_a^b\overline{c}(t)\Delta t\\
&=&\frac{1}{2}\|u\|^2-\overline{a}\|u\|_{L_\Delta^2}^2-\left(\int_a^b|\overline{b}(t)|^{\frac{2}{\gamma}}\Delta t\right)^{\frac{\gamma}{2}}\|u\|_{L_\Delta^2}^{2-\gamma}-\int_a^b\overline{c}(t)\Delta t\\
&\geq&\frac{1}{2}\|u\|^2-\frac{\overline{a}b^{2\alpha}}{\Gamma^2(\alpha+1)}\|u\|^2-\left(\int_a^b|\overline{b}(t)|^{\frac{2}{\gamma}}\Delta t\right)^{\frac{\gamma}{2}}\left(\frac{b^\alpha}{\Gamma(\alpha+1)}\right)^{2-\gamma}\|u\|^{2-\gamma}-\int_a^b\overline{c}(t)\Delta t\\
&=&\left(\frac{1}{2}-\frac{\overline{a}b^{2\alpha}}{\Gamma^2(\alpha+1)}\right)\|u\|^2-\left(\int_a^b|\overline{b}(t)|^{\frac{2}{\gamma}}\Delta t\right)^{\frac{\gamma}{2}}\left(\frac{b^\alpha}{\Gamma(\alpha+1)}\right)^{2-\gamma}\|u\|^{2-\gamma}-\int_a^b\overline{c}(t)\Delta t.
\end{eqnarray*}}
Noting that $\overline{a}\in\left[0,\frac{\Gamma^2(\alpha+1)}{2b^{2\alpha}}\right)$ and $\gamma\in(0,2)$, we obtain $\varphi(u)=+\infty$ as $\|u\|\rightarrow\infty$, and so $\varphi$ is coercive, which completes the proof.
\end{proof}

Let $a_0=\min\limits_{\lambda\in[\frac{1}{2},1]}\left\{\frac{\Gamma^2(\lambda+1)}{2b^{2\lambda}}\right\}$. The following result follows immediately from Theorem \ref{08.1}.

\begin{corollary}For
$\alpha\in\left(\frac{1}{2},1\right]$, if $F$ satisfies the condition $(\mathbf{H_1})$ and \eqref{008.1} with $\overline{a}\in[0,a_0)$, then FBVP \eqref{4.1} has at least one weak solution which minimuzes $\varphi$ on $W^{\alpha,2}_{\Delta,b^-}$.
\end{corollary}

Our task is now to use Theorem \ref{08.1} (Mountain pass theorem) to find a nonzero critical point of functional $\varphi$ on $W^{\alpha,2}_{\Delta,b^-}$.

\begin{theorem}\label{09.1}
Let $\frac{1}{2}<\alpha\leq1$, and suppose that $F$ satisfies $(\mathbf{H_1})$. If
\begin{itemize}
  \item [$(H_2)$]
  $F\in C(J\times\mathbb{R}^N,\mathbb{R})$, and there are $\mu\in\left[0,\frac{1}{2}\right)$ and $M>0$ such that $0<F(t,x)\leq \mu(\nabla F(t,x),x)$ for all $x\in\mathbb{R}^N$ with $\vert x\vert\geq M$ and $t\in J$,
  \item [$(H_3)$]
  $\limsup\limits_{|x|\rightarrow0}\frac{F(t,x)}{|x|^2}<\frac{\Gamma^2(\alpha+1)}{2b^{2\alpha}}$ uniformly for $t\in J$ and $x\in\mathbb{R}^N$
\end{itemize}
are satisfied, then FBVP \eqref{4.1} has at least one nonzero weak solution on  $W^{\alpha,2}_{\Delta,b^-}$.
\end{theorem}

\begin{proof}
We will prove that $\varphi$ satisfies all the conditions of Theorem \ref{08.1}.

First, we will verify that $\varphi$ satisfies P.S. condition. Since $F(t,x)- \mu(\nabla F(t,x),x)$ is continuous for $t\in J$ and $\vert x\vert\leq M$, there is $c\in \mathbb{R}^+$ such that
\begin{eqnarray*}
F(t,x)\leq \mu(\nabla F(t,x),x)+c,\quad t\in J,\,\,\vert x\vert\leq M.
\end{eqnarray*}
In view of condition $(\mathbf{H_2})$, one has
\begin{eqnarray}\label{090.1}
F(t,x)\leq \mu(\nabla F(t,x),x)+c,\quad t\in J,\,\,x\in\mathbb{R}^N.
\end{eqnarray}
Let $\{u_k\}\subset W^{\alpha,2}_{\Delta,b^-}$, $\vert\varphi(u_k)\vert\leq K$, $k=1,2,\cdots$, $\varphi'(u_k)\rightarrow0$. Notice that
\begin{align}\label{5.1}
\langle\varphi'(u_k),u_k\rangle
=&\int_a^b\left[(^\mathbb{T}_tD^\alpha_bu_k(t),\,^\mathbb{T}_tD^\alpha_bu_k(t))-(\nabla F(t,u_k(t)),u_k(t))\right]\Delta t\nonumber\\
=&\|u_k\|^2-\int_a^b\nabla F(t,u_k(t)),u_k(t))\Delta t.
\end{align}
It follows from $(\ref{090.1})$ and $(\ref{5.1})$ that
\begin{align*}
K\geq&\varphi(u_k)\\
=&\frac{1}{2}\int_a^b\vert{^\mathbb{T}_tD^\alpha_b}u_k(t)\vert^2\Delta t-\int_a^bF(t,u_k(t))\Delta t\\
\geq&\frac{1}{2}\|u_k\|^2-\mu\int_a^b(\nabla F(t,u_k(t)),u_k(t))\Delta t-cb\\
=&\left(\frac{1}{2}-\mu\right)\|u_k\|^2+\mu\langle\varphi'(u_k),u_k\rangle-cb\\
\geq&\left(\frac{1}{2}-\mu\right)\|u_k\|^2-\mu\|\varphi'(u_k)\|\|u_k\|-cb.
\end{align*}
Since $\varphi'(u_k)\rightarrow0$, there is $N_0\in \mathbb{N}$ such that
\begin{eqnarray*}
K\geq\left(\frac{1}{2}-\mu\right)\|u_k\|^2-\|u_k\|-cb,\quad k>N_0,
\end{eqnarray*}
which implies that $\{u_k\}\subset W^{\alpha,2}_{\Delta,b^-}$ is bounded. Since $W^{\alpha,2}_{\Delta,b^-}$ is a reflexive space, going to a subsequence if necessary, we may suppose that $u_k\rightharpoonup u$ weakly in $W^{\alpha,2}_{\Delta,b^-}$, therefore, one obtains
{\setlength\arraycolsep{2pt}
\begin{eqnarray}\label{50.1}
&&\langle\varphi'(u_k)-\varphi'(u),u_k-u\rangle\nonumber\\
&=&\langle\varphi'(u_k),u_k-u\rangle-\langle\varphi'(u),u_k-u\rangle\nonumber\\
&\leq&\|\varphi'(u_k)\|\|u_k-u\|-\langle\varphi'(u),u_k-u\rangle\rightarrow0,
\end{eqnarray}}
as $k\rightarrow\infty$. Furthermore, in view of $(\ref{34})$ and Proposition $\ref{39}$, one can get   that $u_k$ is bounded in $C(J,\mathbb{R}^N)$ and $\|u_k-u\|_\infty=0$ as $k\rightarrow\infty$. As a result, one has
\begin{eqnarray}\label{050.1}
\int_a^b\nabla F(t,u_k(t))\Delta t\rightarrow \int_a^b\nabla F(t,u(t))\Delta t,\quad k\rightarrow\infty.
\end{eqnarray}
Noting that
{\setlength\arraycolsep{2pt}
\begin{eqnarray*}
&&\langle\varphi'(u_k)-\varphi'(u),u_k-u\rangle\nonumber\\
&=&\int_a^b(^\mathbb{T}_tD^\alpha_bu_k(t)-^\mathbb{T}_tD^\alpha_bu(t))^2\Delta t-\int_a^b(\nabla F(t,u_k(t))-\nabla F(t,u(t)))\\
&&\times(u_k(t)-u(t))\Delta t\\
&\geq&\|u_k-u\|^2-\left\vert\int_a^b(\nabla F(t,u_k(t))-\nabla F(t,u(t)))\Delta t\right\vert\|u_k-u\|_\infty.
\end{eqnarray*}}
Combining with ($\ref{50.1}$) and $(\ref{050.1})$, it is easy to prove that $\|u_k-u\|^2\rightarrow0$ as $k\rightarrow\infty$, and so that $u_k\rightarrow u$ in $W^{\alpha,2}_{\Delta,b^-}$. Hence, we get the desired convergence property.

By condition $(\mathbf{H_3})$, there are $\epsilon\in(0,1)$ and $\delta>0$ such that $F(t,x)\leq(1-\epsilon)\left(\frac{\Gamma^2(\alpha+1)}{2b^{2\alpha}}\right)\vert x\vert^2$ for all $t\in J$ and $x\in\mathbb{R}^N$ with $\vert x\vert\leq\delta$.

Let $\rho=\frac{\Gamma(\alpha)(2(\alpha-1)+1)^{\frac{1}{2}}}{b^{\alpha-\frac{1}{2}}}\delta$ and $\sigma=\frac{\epsilon\rho^2}{2}>0$. Then it follows from ($\ref{34}$) that
\begin{eqnarray*}
\|u\|_\infty\leq\frac{b^{\alpha-\frac{1}{2}}}{\Gamma(\alpha)(2(\alpha-1)+1)^{\frac{1}{2}}}\|u\|=\delta
\end{eqnarray*}
for all $u\in W^{\alpha,2}_{\Delta,b^-}$ with $\|u\|=\rho$. Hence, combining with ($\ref{33}$), one gets
\begin{align*}
\varphi(u)=&\frac{1}{2}\int_a^b\vert{^\mathbb{T}_tD^\alpha_b}u(t)\vert^2\Delta t-\int_a^bF(t,u(t))\Delta t\\
=&\frac{1}{2}\|u\|^2-\int_a^bF(t,u(t))\Delta t\\
\geq&\frac{1}{2}\|u\|^2-(1-\epsilon)\frac{\Gamma^2(\alpha+1)}{2b^{2\alpha}}\int_a^b\vert u(t)\vert^2\Delta t\\
\geq&\frac{1}{2}\|u\|^2-\frac{1}{2}(1-\epsilon)\|u\|^2\\
=&\frac{1}{2}\epsilon\|u\|^2\\
=&\sigma
\end{align*}
for all $u\in W^{\alpha,2}_{\Delta,b^-}$ with $\|u\|=\rho$. This implies $(\mathbf{ii})$ in Theorem \ref{08.1} is satisfied.

It is obvious from the definition of $\varphi$ and condition $(\mathbf{H_3})$ that $\varphi(0)=0$, and so, it suffices to prove that $\varphi$ satisfies $(\mathbf{iii})$ in Theorem \ref{08.1}.

Since condition $(\mathbf{H_2})$, a simple regularity argument then proves that there is $r_1,r_2>0$ such that
\begin{eqnarray*}
F(t,x)\geq r_1\vert x\vert^{\frac{1}{\mu}}-r_2,\quad x\in\mathbb{R}^N,\,\,t\in J.
\end{eqnarray*}
For any $u\in W^{\alpha,2}_{\Delta,b^-}$ with $u\neq0$, $\kappa>0$ and noting that $\mu\in\left[0,\frac{1}{2}\right)$, one has
\begin{align*}
\varphi(\kappa u)=&\frac{1}{2}\int_a^b\vert{^\mathbb{T}_tD^\alpha_b}\kappa u(t)\vert^2\Delta t-\int_a^bF(t,\kappa u(t))\Delta t\\
\leq&\frac{1}{2}\|\kappa u\|^2-r_1\int_a^b\vert\kappa u(t)\vert^{\frac{1}{\mu}}\Delta t+r_2b\\
=&\frac{1}{2}\kappa^2\|u\|^2-r_1\kappa^{\frac{1}{\mu}}\|u\|_{K_\Delta^{\frac{1}{\mu}}}^{\frac{1}{\mu}}+r_2b\\
\rightarrow&-\infty
\end{align*}
as $k\rightarrow\infty$. Then there is a sufficiently large $\kappa_0$ such that $\varphi(\kappa_0u)\leq0$. As a result, $(\mathbf{iii})$ of Theorem \ref{08.1} holds.

Lastly, noting that $\varphi(0)=0$ while for our critical point $u$, $\varphi(u)\geq\sigma>0$. Therefore, $u$ is a nontrivial weak solution of the FBVP \eqref{4.1}, and this completes the proof.
\end{proof}

\begin{corollary}
For all $\alpha\in\left(\frac{1}{2},1\right]$, assume that $F$ satisfies conditions $(\mathbf{H_1})$ and $(\mathbf{H_2})$. If
\begin{itemize}
  \item [$(H_3)'$]
  $F(t,x)=o(\vert x\vert^2)$, as $\vert x\vert\rightarrow0$ uniformly for $t\in J$ and $x\in\mathbb{R}^N$
\end{itemize}
is satisfied, then the FBVP \eqref{4.1} has at least one nonzero weak solution on $W^{\alpha,2}_{\Delta,b^-}$.
\end{corollary}

{}

\end{document}